\documentclass[a4paper, 11pt]{article}
\usepackage[a4paper,hmargin={2.5cm,2.5cm},vmargin={3cm,3cm}]{geometry}
\usepackage{amsmath,amssymb,amsthm}
\usepackage{hyperref}

\usepackage{graphicx}

\usepackage{wrapfig}
\usepackage{tikz}
\usetikzlibrary{arrows,calc,decorations.pathreplacing}
\definecolor{light-gray1}{gray}{0.90}
\definecolor{light-gray2}{gray}{0.80}
\definecolor{light-gray3}{gray}{0.60}

\title{Dispersive estimates for rational symbols and local well-posedness of the nonzero energy NV equation}
\date{}
\author{Anna Kazeykina\footnote{Laboratoire de Math\'ematiques d'{}Orsay UMR 8628, B\^at. 425 Facult\'e des Sciences d'{}Orsay,
Universit\'e Paris-Sud F-91405 Orsay Cedex France. E-mail: anna.kazeykina at math.u-psud.fr} \quad and \quad  Claudio Mu\~noz\footnote{CNRS and Laboratoire de Math\'ematiques d'{}Orsay UMR 8628, B\^at. 425 Facult\'e des Sciences d'{}Orsay,
Universit\'e Paris-Sud F-91405 Orsay Cedex France. E-mail: claudio.munoz at math.u-psud.fr}}
\begin{document}

\maketitle
\markboth{A}{Anna Kazeykina and Claudio Mu\~noz}
\renewcommand{\sectionmark}[1]{}
\begin{abstract}
We consider the Cauchy problem for the two-dimensional Novikov-Veselov equation integrable via the inverse scattering problem for the Schr\"odinger operator with fixed negative energy. The associated linear equation is characterized by a rational symbol which is not a polynomial, except when the energy parameter is zero. With the help of a complex analysis point of view of the problem, we establish uniform decay estimates for the linear solution with gain of almost one derivative, and we use this result together with Fourier decomposition methods and $X^{s,b}$ spaces to prove local well-posedness in $H^s$, $s>\frac12$.
\end{abstract}


\newtheorem{thm}{Theorem}[section]
\newtheorem{cor}[thm]{Corollary}
\newtheorem{lem}[thm]{Lemma}
\newtheorem{prop}[thm]{Proposition}
\theoremstyle{definition}
\newtheorem{defn}{Definition}[section]

\theoremstyle{remark}
\newtheorem{rem}{Remark}[section]
\newtheorem*{notation}{Notation}
\newtheorem{Cl}{Claim}

\numberwithin{equation}{section}

\newcommand{\const}{\mathrm{const}}
\newcommand{\sgn}{\mathrm{sgn}}
\renewcommand{\Re}{\mathrm{Re}}
\renewcommand{\Im}{\mathrm{Im}}
\newcommand{\re}{\operatorname{Re}}
\newcommand{\ima}{\operatorname{Im}}
\newcommand{\interior}{\mathrm{int}}
\newcommand{\meas}{\operatorname{meas}}
\newcommand{\R}{\mathbb{R}}
\newcommand{\N}{\mathbb{N}}
\newcommand{\Z}{\mathbb{Z}}
\newcommand{\T}{\mathbb{T}}
\newcommand{\Q}{\mathbb{Q}}
\newcommand{\Com}{\mathbb{C}}
\newcommand{\la}{\lambda}
\newcommand{\pd}{\partial}
\newcommand{\al}{\alpha}
\newcommand{\bt}{\beta}
\newcommand{\ga}{\gamma}
\newcommand{\de}{\delta}
\newcommand{\te}{\theta}
\newcommand{\arctanh}{\operatorname{arctanh}}
\newcommand{\spawn}{\operatorname{span}}
\newcommand{\sech}{\operatorname{sech}}
\newcommand{\dist}{\operatorname{dist}}
\newcommand{\diam}{\operatorname{diam}}
\newcommand{\supp}{\operatorname{supp}}
\newcommand{\pv}{\operatorname{pv}}
\newcommand{\vv}[1]{\partial_x^{-1}\partial_y{#1}}

\def\bm{\left( \begin{array}{cc}}
\def\endm{\end{array}\right)}
\providecommand{\abs}[1]{\left|#1 \right|}
\providecommand{\norm}[1]{\left\| #1 \right\|}
\newcommand{\sublim}{\operatornamewithlimits{\longrightarrow}}
 \newcommand{\ex}{{\bf Example}:\ }
\newcommand{\ve}{\varepsilon}
\newcommand{\fin}{\hfill$\blacksquare$\vspace{1ex}}

\newcommand{\be}{\begin{equation}}
\newcommand{\ee}{\end{equation}}
\newcommand{\ba}{\begin{equation*}}
\newcommand{\ea}{\begin{equation*}}
\newcommand{\bea}{\begin{eqnarray}}
\newcommand{\eea}{\end{eqnarray}}
\newcommand{\bee}{\begin{eqnarray*}}
\newcommand{\eee}{\end{eqnarray*}}
\newcommand{\ben}{\begin{enumerate}}
\newcommand{\een}{\end{enumerate}}
\newcommand{\nonu}{\nonumber}

\setlength\arraycolsep{2pt}

\section{Introduction. Main results}

\medskip

In this paper we consider the initial value problem associated to the Novikov-Veselov (NV) equation, in two dimensions, with a fixed energy parameter $E\in \R$ \cite{NV}:
\bea
\partial_t v  & =& 8 (\partial_z^3 + \partial^3_{\bar z}) v +2\partial_z (vw) +2 \partial_{\bar z}(v \bar w) -2E(\partial_z w + \partial_{\bar z} \bar w ), \label{NV1} \\
\partial_{\bar z} w& =&  -3 \partial_z v \label{NV2}.
\eea
Here $v=v(t,x,y)$, $w=w(t,x,y)$, and
\[
z=x+iy =(x,y) \in \Com.
\]
The complex partial derivative operators $\partial_z$ and $\partial_{\bar z}$ are defined as usual:
\[
\partial_z := \frac12 (\partial_x - i \partial_y),  \quad \partial_{\bar z} := \frac12 (\partial_x + i \partial_y), \quad (\Delta_{x,y} =4 \partial_z\partial_{\bar z} =4 \partial_{\bar z}\partial_z).
\]
Formally, one has  $w = -3\partial_{\bar z}^{-1} \partial_z v$, where $\partial_{\bar z}^{-1} \partial_z$ can be defined via the Fourier transform $\mathcal F$, acting on $L^2(\R^2)$, $\xi=(\xi_1,\xi_2)\in\R^2\backslash\{(0,0)\}$:
\be\label{W_V}
\mathcal F[\partial_{\bar z}^{-1} \partial_z f] (\xi_1,\xi_2)= \Big(\frac{\xi_1 -i \xi_2}{\xi_1 + i \xi_2}\Big)\mathcal F[f](\xi_1,\xi_2).
\ee
In addition, note that the right hand side in \eqref{NV1} has the structure of `a real part'. Consequently, one looks for $v$ real-valued:
\[
v(t,x,y) \in \R, \quad w(t,x,y) \in \Com,
\]
so that the first equation can be written as follows:
\[
\partial_t v = 4 \, \Re \{ \partial_{z} (4 \, \partial_z^2 v - 3\, v\partial_{\bar z}^{-1} \partial_z v +3\, E\partial_{\bar z}^{-1} \partial_z v) \} .
\]
For the sake of completeness, we write equations \eqref{NV1}-\eqref{NV2} in terms of real-valued coordinates $(x,y)$. If we identify $w=w_1+iw_2 $ with the vector field $w=(w_1,w_2)$, with $w_1,w_2$ real-valued, then \eqref{NV2} becomes
\[
\partial_y w_1+\partial_x w_2 =3~\partial_y v, \quad \partial_y w_2- \partial_x w_1 =3~ \partial_x v,
\]
and \eqref{NV1} reads
\be\label{NV_xy}
\partial_t v = 2\left[ \partial_x (\partial_x^2 v-3 \partial_y^2v) +\nabla .(vw) -E\nabla .w \right].
\ee
See Appendix \ref{A} for a proof of this fact. However, for most of this paper we will work with the $z$-$\bar z$ formulation \eqref{NV1}-\eqref{NV2}, which presents several advantages when computing precise estimates.

\medskip

The NV equations are important for several reasons. First of all, \eqref{NV1}-\eqref{NV2} is a \emph{completely integrable} model in two dimensions, in the spirit of the Kandomtsev-Petviashvili (KP) and the Davey-Stewartson (DS) equations \cite{KBook}. However, unlike the two last models, the origin of NV is not physical, but mathematical \cite{NV}: it arises as the integrable equation obtained by assuming that the  associated scattering problem corresponds to the standard, stationary Schr\"odinger equation in two dimensions with fixed energy $E$:
\be\label{Schrodinger}
( -\Delta_{x,y} + v(t,x,y) -E ) \psi = 0.
\ee
In that sense, NV is the closest model that generalizes (in a mathematical form) the famous Korteweg-de Vries (KdV) equation to the two dimensional case, and it has interesting connections with the scattering problem for bounded or localized potentials, and the theory of  inverse problems for sufficiently smooth potentials, see e.g. \cite{Perry} and references therein for more details. Moreover, if $v$ depends only on $x$, we have $z=\bar z=x$, $\partial_z =\partial_{\bar z} =\frac12 \partial_x$ and
\[
w = -3 v,  \quad \partial_t v = 2   \partial_{x} (\partial_x^2 v - 3 v^2 + 3 E v),
\]
which is a KdV equation.\footnote{Note that in this case $u(t,x):=-v(-t/2 , x - 3Et)$ solves the standard KdV equation $u_t + u_{xxx} +6uu_x=0.$}

\medskip

Another interesting connection between NV and other integrable models is the following.
It can be formally shown (see \cite{Gr2} for example and Appendix \ref{B} for more details) that when the parameter of energy $ E $ tends to $ \pm \infty $ the NV equations, rescaled appropriately, become in the limit KP equations (KPI corresponding to $ E \to +\infty$ and KPII corresponding to $ E \to -\infty $).  It was also shown formally in \cite{Gr2} that there is a a corresponding convergence of scattering problems for the KP and NV equations. In order to obtain a rigorous proof of such asymptotic regime, a suitable well-posedness theory for NV equations is necessary, with explicit bounds of the solution and its lifespan depending on $E$.

\medskip

The Novikov-Veselov equation was first obtained in an implicit form by S.V. Manakov in \cite{Manakov}.
It has the following operator representation
\be\label{Manakov_triple_0}
\partial_t  L = [L,A] +BL,
\ee
where
\[
L := -\Delta_{x,y} + v(t,x,y)  -E, \quad A:= -8(\partial_z^3 +\partial_{\bar z}^3)  -2(w \partial_z +\bar w \partial_{\bar z}),
\]
and
\[
B:= 2(\partial_z w + \partial_{\bar z}\bar w),
\]
with $ w $ given by (\ref{NV2}) and $[\cdot,\cdot]$ denoting the commutator. Equation (\ref{Manakov_triple_0}) is the compatibility condition of
\[
L \varphi =0, \quad \varphi_t = A\varphi.
\]

Representation (\ref{Manakov_triple_0}) is called Manakov triple representation for (\ref{NV1}) and can be considered as a generalization of the Lax pair representation for KdV (see \cite{Lax}) to the $ ( 2 + 1 ) $-dimensional case.
\medskip

From the integrability of \eqref{NV1}-\eqref{NV2} it follows that NV formally has an infinite number of conserved quantities. Here we write the first three that have a correct meaning (each integral is taken on $\Com \cong \R^2$): the $L^1$ integral
\be\label{L1}
\int v(t) dxdy =\int v(0)dxdy, \quad \hbox{(real-valued),}
\ee
the ``mass''
\bea\label{Mass}
M[v](t) & :=& \int vw(t) dxdy \nonu\\
& =& \int vw(0)dxdy, \quad \hbox{(complex-valued, two identities),}
\eea
and the ``energy'':
\bea\label{Energy}
H[v](t) & := & \int \Big[  6\,  \partial_z w \partial_z v  + E w^2 - vw^2  \Big] (t)dxdy \nonu\\
& =& H[v](0),\qquad \hbox{(complex-valued, two identities).}
\eea
We see that  $H^1(\R^2)$ is the energy space, for if $v\in H^1$, then $\partial_z v$ and $\partial_{\bar z} v$ are in $L^2$, and $w \in H^1$. It is not difficult to check that we also have\footnote{Indeed, this fact is a consequence of the Gagliardo-Nirenberg inequality
\be\label{GN}
\|u\|_{L^p} \leq C\|\nabla u\|_{L^2}^{\al} \|u\|_{L^2}^{1-\al}, \quad \al\in [0,1), ~ p=\frac2{1-\al}.
\ee}
\[
 v, w\in L^3(\R^2).
 \]
Compared with other dispersive models coming from Physical theories, NV equations lack of signed conserved quantities which control the long time dynamics. For this reason, a good understanding of some particular explicit solutions is essential. These regimes will strongly depend on the sign of the energy parameter $ E $: we will see that $E<0$ corresponds to a sort of defocusing case, whilst $E>0$ is the focusing regime. The case $E=0$ is different in several aspects.

\medskip

Indeed, the literature \cite{KBook} describes three different regimes associated to \eqref{NV1}-\eqref{NV2}, depending on the value of the energy $E$: the case where $E=0$ (usually denoted as $NV_0$), and the cases where $E>0$ ($NV_+$) or $E<0$ ($NV_-$).

\medskip

In the case $E=0$, using direct and inverse scattering techniques, Perry \cite{Perry2} has been able to show existence of a solution for a certain set of initial data. Angelopoulous \cite{A} showed that the Cauchy problem for \eqref{NV1}-\eqref{NV2} is locally well-posed for $v(t=0)\in H^s(\R^2)$, $s>\frac12$, by following the ideas of Molinet and Pilod \cite{MP} for the Zakharov-Kuznetsov equation. Below $L^2(\R^2)$ he showed \cite{A} that the flow has some ill-posed behavior. One may ask if by using \eqref{Mass} or \eqref{Energy} the well-posedness result may be extended to a global one. The answer is surprisingly no (another fact that reveals that NV does not enjoy standard physical properties). For any $a,c,d\in \R$ such that $a+c (x^3+y^3) +d(x^2+y^2)^2>0$ everywhere, the function
\[
v(t,x,y)= -2 \Delta_{x,y} \log (a -24c t +c (x^3+y^3) +d(x^2+y^2)^2)
\]
solves $NV_0$, decays like $r^{-3}$ at infinity ($r=\sqrt{x^2+y^2}$), and it blows up at finite time (see also \cite{TS2008}).
Note additionally that if $v(t,x,y)$ is solution of NV, then for all $\la>0$,
\be\label{Sca1}
v_\la(t,x,y) := \la^2 v(\la^3 t,\la x,\la y), \quad \hbox{and} \quad w_\la(t,x,y) := \la^2 w(\la^3 t,\la x,\la y)
\ee
satisfy \eqref{NV1}-\eqref{NV2} with
\be\label{Sca2}
E_\la := E \la^2.
\ee
In the case $E=0$ the equation does not change and we have another solution of the same equation.\footnote{Therefore, the critical space for $E=0$ is $H^{-1}$.}  Using this scaling symmetry, we can construct a solution that blows up at every time of the form $t_\la$, 
$\la\neq 0$, at the same points as before. Moreover, since $H^1$ or even $L^2$ are subcritical regularities, one realizes that  there are arbitrarily small solutions that blow up in finite time, unlike in standard dispersive models. In addition, the time of existence will not depend on the size of the solution only.

\medskip



In the case of negative energy, $NV_-$ is in some sense reminiscent of the KPII equation studied by Bourgain in \cite{B1}.
It has been shown in \cite{K2} via the inverse scattering techniques that in this case the equation does not possess sufficiently regular soliton solutions decaying as $ r^{-3 - \varepsilon} $, $ \varepsilon > 0 $ (compare with results of \cite{BS1}, \cite{BS2}, \cite{BM} on the absence of sufficiently localized solutions of KPII). The inverse scattering theory also allowed to obtain the large time asymptotics (uniform in space) for sufficiently regular solutions of $ NV_- $ with nonsingular scattering data (in particular, sufficiently small solutions have nonsingular scattering data) (see \cite{K} and compare with an analogous uniform large-time asymptotic result for KPII in \cite{Kis}). In this paper, we prove the following result.

\begin{thm}\label{LWP_neg}
$NV_-$ is locally well-posed in $H^s(\R^2)$ for any $s>\frac 12$. The lifespan of solution is proportional to the size of the energy $|E|$.
\end{thm}
The precise version of this Theorem is formulated as Thereom \ref{Cauchy} below. Here we present briefly the ideas of the proof of this result.

\medskip

In the proof of Theorem \ref{LWP_neg}  we follow the ideas of Bourgain, extended and simplified by many authors; in particular we use the formalism employed by Molinet and Pilod in \cite{MP} fot the study of low-regularity solution to the Zakharov-Kuznetsov (ZK) equation in several geometries in 2 and 3 dimensions. Note that in \cite{MP} the authors deal with symbols whose resonance functions seem a priori intractable via Bourgain's ideas (their method was developed in \cite{A} for NV in the case $ E = 0 $). It is essential in their work that Molinet and Pilod use a sharp $L^4$ Strichartz estimate with gain of derivatives obtained by Carbery, Kenig and Ziesler \cite{CKZ} for polynomial symbols like the one in ZK. In our case, we do not have a polynomial symbol, but a fractional one, much in the spirit of KP equations. We overcome the difficulty of showing a Carbery-Kenig-Ziesler type estimate by proving almost sharp smoothing estimates in two dimensions for the linear part of the equation, in the spirit of Kenig, Ponce and Vega \cite{KPV_Indiana}, and Saut \cite{Saut}, a task that requires some fine Fourier analysis coming from the improvement of previous results showed by the first author \cite{KN,K} in the framework of the Inverse Scattering  Transformation approach. It turns out that we gain almost one derivative in $L^\infty$ with decay slighty slower than $1/t$, which suffices for us to close standard Bourgain's bilinear estimates for $s>\frac12$.

\medskip

In addition to these estimates, we perform a standard Fourier localization between low and high frequencies. We recall that the resonance function that appears when dealing with the interaction of low-high to high frequencies is treated by estimating the zero level set via reasonable lower bounds on the partial derivatives of the resonance function. Such estimates are simple to establish, but they carry a loss of accuracy (probably 1/2 of derivative) that could be avoided by dealing directly with the resonance function as Bourgain did in \cite{B1}; however, such a task substantially more difficult given the complexity of the linear NV symbol (see \eqref{w0} for more details). However, it is worth to mention that the nonzero energy NV symbol has some useful boundedness properties near the origin, unlike standard KP equations. Finally, the high-high to high frequencies interactions are treated using the almost sharp smoothing estimate for the linear symbol, which leads to a Strichartz  estimate with $1/4^-$ gain of derivative in $L^4$.

\medskip
%

We finish this introductory section with some words about the global existence problem. Although the equation is in nature well-behaved for high regularity initial data, the problem of global existence is far from trivial because of the lack of an evident sign in the real and imaginary part of the conservation laws, see \eqref{Mass} and \eqref{Energy}. Even worse, some nice initial data may have bad behavior in terms of conserved quantities: for example any smooth, rapidly decaying radial data has radial Fourier transform and consequently it mass \eqref{Mass} is identically zero:
\[
\int vw = \int \overline{\hat v}\hat w = \int_0^\infty\int_0^{2\pi} e^{-2i\theta} |\hat v(r)|^2 r\ dr d\theta =0.
\]
This implies that in principle the mass cannot be used to bound the $L^2$ norm of the solution, as usually performed in other physical models. A similar conclusion holds for the two first terms in the Hamiltonian \eqref{Energy}.

\medskip

It turns out that in the case of NV equations, such a problem is deeply related with the behavior of scattering solutions of the associated Schr\"odinger operator \eqref{Schrodinger}. The results of Novikov \cite{Nov} imply global existence for initial data with suitable spectral properties, via the inverse scattering transform. Very recently, Perry \cite{Perry2} showed that in the case of zero energy, if the initial data are of conductivity type, in addition to other decay assumptions, then $NV_0$ has a global weak solution that is classical if the initial data are in the Schwartz class. Recall that a potential $v_0 \in C_0^\infty (\R^2)$ is of conductivity type if the equation
\[
(-\Delta_{x,y} + v_0) \varphi=0
\]
admits a unique positive solution such that $\varphi \equiv 1$ in a neighborhood of infinity. This definition comes from the theory of Inverse Problems, see \cite{Calderon,SU,Nachman} for more details on this interesting connection.

\medskip

Finally, some words about the case $E>0$. The methods of this paper are not entirely suitable for the case $E>0$. The main obstacle is the fact that the smoothing Strichartz estimate is not longer sufficient to control the large-to-large frequencies regime, and an approach of the type Carbery-Kenig-Ziesler is probably needed. We hope to consider the remaining case in a forthcoming publication.

\medskip

{\bf Acknowledgments.} The second author would like to thank the Laboratoire de Math\'e-matiques d'{}Orsay for their kind hospitality during past years, and Didier Pilod for explaining his paper \cite{MP} and giving several useful comments about a preliminary manuscript. The authors also thank Jean-Claude Saut for the initial motivation of this study and for many useful discussions.

\medskip

{\bf Notation.} In the following text notation ``$ A \lesssim B $'' means that there exists a constant $ c $ (depending only on $ \alpha $) such that $ A \leq c B $.

\medskip

\section{Smoothing estimates for negative energies}

The aim of this section is to estimate the integral
\be\label{I}
I (t,u;E):= \int\limits_{ \mathbb{C} } | \xi |^{ \alpha+i\bt }  e^{ i t \tilde S( u, \xi ;E) } d \Re \xi d \Im \xi,
\ee
where
\[
 u = \frac{z}t,
\]
and
\begin{equation}
\label{st_phase}
\tilde S( u, \xi ; E) = ( \xi^3 + \bar \xi^3 )\left( 1 - \frac{ 3 E }{ |\xi|^2 } \right) + \frac{ 1 }{ 2 }( \bar u \xi + u \bar \xi ).
\end{equation}
Note that
\begin{equation*}
p( \xi ) = ( \xi^3 + \bar \xi^3 )\left( 1 - \frac{ 3 E }{ |\xi|^2 } \right)
\end{equation*}
is the symbol of the linear part of the NV equation.


\begin{lem}[Dispersion estimate for negative energy]
\label{asp_lin_estimate_lemma}
Let $ 0 \leq \alpha < 1 $ and $\bt \in \R$. Then for all $ t > 0 $ and for any fixed $\ve(\al)>0$ small, the following estimate is valid:
\begin{equation}\label{I_decrease}
| I (t,u;-1)| \lesssim  \frac{(1+|\bt|) }{ t^{(\alpha + 3)/4 -\ve } },
\end{equation}
uniformly on $ u \in \mathbb{C} $. Recall that the implicit constant depends on $\al$ only.
\end{lem}


\begin{rem}
Note that we obtain better estimates than in the case of the Zakharov-Kuznetsov's oscillatory integral, see e.g.  \cite{LP}. Compared to the results in \cite{LP}, when $\al\sim 1$ we gain almost one derivative instead of half a derivative.
\end{rem}

\begin{rem}
By making a suitable change of variables, we easily prove a dispersion estimate in the case of arbitrary negative energy $E<0$. We have
\be\label{Ineq_E}
| I (t,u;E)| \, \lesssim  \,  \frac{(1+|\bt|) |E|^{(\al-1)/8 -3\ve/2} }{ t^{ (\alpha + 3 )/4-\ve }  },
\ee
with constants independent of $E$. Indeed, given \eqref{I} with phase \eqref{st_phase}, we are reduced to the case $E=-1$ by performing the change of variables  $\xi= |E|^{1/2} \xi_0 $, $t_0= |E|^{3/2}t$, and $ z_0=|E|^{1/2} z$, which converts \eqref{I} into
\[
\int\limits_{ \mathbb{C} } |E|^{\al/2}| \xi _0|^{ \alpha }  e^{ i t_0 \tilde S( u_0, \xi_0;-1) } |E| d \Re \xi_0 \, d \Im \xi_0,\quad u_0 := \frac{z_0}{t_0}.
\]
Using \eqref{I_decrease} we are lead to
\bee
|I(t,u;E)|  & \lesssim  & \frac{  (1+|\bt|) |E|^{\frac{\al+2}{2}} }{ t_0^{\frac{\al+3}4 -\ve}} \\
& \sim & (1+|\bt|)  |E|^{\frac{\al+2}{2}}  \frac{ |E|^{-\frac{3(\al+3)}{8} -\frac32\ve}}{ t^{ \frac{\alpha + 3}4-\ve } } \\
&  \lesssim & (1+|\bt|) \frac{ |E|^{\frac{\al-1}8 -\frac 32\ve} }{ t^{ \frac{\alpha + 3}4 -\ve} },
\eee
as desired. Note that  in the formal limit $E\to 0$, estimate \eqref{Ineq_E} becomes singular.
\end{rem}
\medskip
Some useful consequences of Lemma \ref{asp_lin_estimate_lemma} are stated below.

\begin{cor}[Smoothing and Strichartz estimates]
Let $U(t)=U(t;E)$ be the associated $NV_-$ linear group, namely for $\xi =\xi_1 +i\xi_2$,
\[
U(t)v_0 :=  \int_{\R^2} e^{ i t \tilde S( u, \xi ,E) } \hat{v_0}(\xi_1,\xi_2)d\xi_1 d\xi_2,
\]
$($cf. \eqref{I} and \eqref{st_phase}$)$. Then for any $0\leq \al< 1$, $\bt\in [0,1]$, we have the smoothing decay estimate
\be\label{Dispersion}
\| |\partial_z|^{\al \bt} U(t)v_0 \|_{L^p_{x,y}} \lesssim_{\al,\bt}  \frac{ |E|^{\bt \big( \frac{\al-1}8- \frac32 \ve\big)} }{ t^{ \bt\big(\frac{\alpha + 3 }4 - \ve \big)} } \|v_0\|_{L^{p'}_{x,y}},
\ee
and the Strichartz estimate with gain of almost one-half derivative:
\be\label{Smoothing}
\| |\partial_z|^{\al \bt/2} U(t)v_0 \|_{L^q_tL^p_{x,y}} \lesssim_{\al,\bt}   |E|^{ \bt\big(\frac{\al-1}{8}-\frac32\ve\big)} \|v_0\|_{L^{2}_{x,y}},
\ee
with $p=\frac{2}{1-\beta}$, $\frac1p+\frac1{p'}=1$, and $\frac 2q= \bt \big( \frac{\alpha + 3}4 -\ve \big)$.
\end{cor}
\begin{proof}
The proofs  are by now standard, but we prove it for the sake of completeness. See e.g. \cite{SW,KPV_Indiana} and references therein for detailed proofs.

\medskip

First,  we prove \eqref{Dispersion}. This is just a consequence of the standard complex interpolation theorem. We have from \eqref{Ineq_E} and Young's inequality,
\[
\| |\partial_z|^{\al +i\bt} U(t)v_0 \|_{L^\infty}  \lesssim  \, \frac{ (1+|\bt|) |E|^{\frac{\al-1}8 -\frac32\ve} }{ t^{ \frac{\alpha + 3}4 -\ve } } \|v_0\|_{L^1},
\]
for any $0\leq \al<1$ and $\bt\in \mathbb{R}$. We interpolate against the trivial estimate
\[
\| |\partial_z|^{i\bt} U(t)v_0 \|_{L^2} = \|v_0\|_{L^2},
\]
to get \eqref{Dispersion}.
\medskip

Now we prove \eqref{Smoothing}. By duality, we are lead to prove that
\[
\int {\phi}(t,x,y)  |\partial_z|^{\al \bt/2} U(t)v_0 \, dtdxdy \lesssim  |E|^{\bt \big(\frac{\al-1}{8} -\frac32 \ve\big)} \|v_0\|_{L^2} \|\phi\|_{L^{q'}_t L_{x,y}^{p'}},
\]
for all $\phi\in C^\infty_0(\R^3)$ and with $ \frac1q + \frac1{q'} = 1 $. However,
\bee
\int {\phi}(t,x,y)  |\partial_z|^{\al \bt/2} U(t)v_0  \, dtdxdy & =&  \int v_0 \Big[ \int {  |\partial_z|^{\al \bt/2}  U(-t)\phi(t,x,y)}  dt \Big]dxdy\\
&\leq &  \|v_0\|_{L^2} \norm{ \int {  |\partial_z|^{\al \bt/2}  U(-t) {\phi(t)}}  dt}_{L^2}.
\eee
Now,
\bee
 \norm{ \int {  |\partial_z|^{\al \bt/2}  U(-t)\phi(t)}  dt}_{L^2}^2 &=& \int_t \!\int_{t'} \!\int_{x,y} {  |\partial_z|^{\al \bt/2}  U(-t')\phi(t',x,y)}    |\partial_z|^{\al \bt/2}  U(-t)\phi(t,x,y) dxdy dt' dt  \\
 & =&  \int_{x,y} \!\int_{t} \!\int_{t'}   |\partial_z|^{\al \bt}  U(t'-t)\phi(t',x,y) \phi(t,x,y)  dt' dt dxdy \\
 & \leq & \norm{ \int_{t'}   |\partial_z|^{\al \bt}  U(t'-\cdot )\phi(t')dt' }_{L^q_t L^p_{x,y}} \|\phi\|_{L^{q'}_t L^{p'}_{x,y}}.
\eee
On the other hand, we have from \eqref{Dispersion} and the Hardy-Littlewood-Sobolev's inequality,
\bee
 \norm{ \int_{t'}   |\partial_z|^{\al \bt}  U(t'-\cdot )\phi(t')dt' }_{L^q_t L^p_{x,y}} & \leq &   \norm{ \int_{t'} \norm{  |\partial_z|^{\al \bt}  U(t' - \cdot)\phi(t') }_{L^p_{x,y}} dt' }_{L^q_t}\\
 & \lesssim &  |E|^{\bt\big(\frac{\al-1}{8} -\frac32\ve\big)} \norm{ \int_{ t' } \frac{  \norm{ \phi(t') }_{L^{p'}_{x,y}} }{ |t-t'|^{ \bt \big(\frac{\alpha + 3}4 -  \ve\big)} } dt' }_{L^q_t}\\
& \lesssim &   |E|^{\bt \big(\frac{\al-1}{8}-\frac32\ve\big)} \norm{ \phi }_{L^{q'}_t L^{p'}_{x,y} }.
\eee
\end{proof}

Another set of standard but useful estimates is the following.
\begin{cor}
We have
\be\label{Smoothing1}
\| U(t)v_0 \|_{L^{4}_{t,x,y}} \lesssim  |E|^{-1/16^+} \| U(t)v_0\|_{X^{0,\frac7{16}+}},
\ee
and
\be\label{Smoothing2}
\| |\partial_z|^{1/4^-}U(t)v_0 \|_{L^{4}_{t,x,y}} \lesssim |E|^{-0^+} \| U(t)v_0\|_{X^{0,\frac12^-}}
\ee
The constant in the last inequality becomes singular as $\frac14^-$ approaches $\frac14.$
\end{cor}
\begin{proof}
Put $\al=0$, $\bt= \big( \frac74-\ve\big)^{-1} =\frac47^+$ in \eqref{Smoothing}, we get
\[
\| U(t)v_0 \|_{L^{r}_{t,x,y}} \lesssim |E|^{ \gamma}  \|v_0\|_{L^{2}_{x,y}},
\]
with $ r = \frac{ 14 }{ 3 } \left( \frac{ 1 - \frac{ 4 }{ 7 } \ve }{ 1 - \frac{ 4 }{ 3 } \ve } \right) = \frac{ 14 }{ 3 }^+ $ and $ \gamma = - \frac{ 1 }{ 14 } \left(\frac{ 1- 12 \ve }{ 1 - \frac{ 4 }{ 7 } \ve } \right) =  -\frac{ 1 }{ 14 }^+ $.

From here we have (see e.g. \cite[Lemma 3.3]{G})
\[
\| U(t)v_0 \|_{L^{r}_{t,x,y}} \lesssim  |E|^{\gamma}  \| U(t)v_0 \|_{X^{0,\frac12+}},
\]
By interpolation with the trivial estimate for $\| v \|_{L^{2}_{t,x,y}}$, using the interpolation parameter $\theta = \frac{ 7 }{ 8 }( 1 - \frac{ 4 }{ 7 } ) = \frac{ 7 }{ 8 }^-$ and $p_\theta^{-1} =(1-\theta)p_0^{-1} +\theta p_1^{-1}$, $0<p_0<p_1$, we obtain \eqref{Smoothing1}.

\medskip

Now we deal with \eqref{Smoothing2}. Taking $\al=1- 16\ve$ and $\bt=\frac{1}{2-5\ve}=\frac12^+$ in \eqref{Smoothing} we have $p=q =4 \frac{1-\frac52 \ve}{1-5\ve} =4^+$. We also have $\frac12\al\bt = \frac{1-16\ve}{4(1-\frac52\ve)} =\frac14^-$ and
\[
\| |\partial_z|^{\frac14^-} U(t)v_0 \|_{L^{r}_{t,x,y}} \lesssim |E|^{-0^+} \|v_0\|_{L^{2}_{x,y}},
\]
with $ r = 4 \frac{1-\frac52 \ve}{1-5\ve} $, from where
\[
\| |\partial_z|^{\frac14^-} U(t)v_0 \|_{L^{r}_{t,x,y}} \lesssim  |E|^{-0^+}   \| U(t)v_0 \|_{X^{0,\frac12+}}.
\]
We interpolate with $L^2_{t,x,y}$ with $\theta= 1 - \frac52 \ve $ to obtain
\[
\| |\partial_z|^{1/4^-} U(t)v_0 \|_{L^{4}_{t,x,y}} \lesssim  |E|^{-0^+} \| U(t)v_0 \|_{X^{0, \frac12^-}}.
\]
\end{proof}

In the next Section we prove Lemma \ref{asp_lin_estimate_lemma}.

\section{Proof of Lemma \ref{asp_lin_estimate_lemma}}

\medskip

\subsection{Restatement of the problem} This section is devoted to the proof of the pointwise dispersion estimate for nonzero energies. Compared to previous works (see e.g. \cite{Saut,LP}), we estimate the whole two dimensional integral. We start by performing the following change of variables\footnote{Recall that we work with $E=-1$. In the case of general $E<0$, the right change of variables is given by
\[
\xi = - i \sqrt{|E|}\left( \lambda - \frac{ 1 }{ \overline \lambda } \right), \quad \bar \xi = i \sqrt{|E|}\left( \bar \lambda - \frac{ 1 }{ \lambda } \right).
\]
Note that this formulation breaks down when $E$ approaches the zero energy level.}:
\begin{equation}\label{ch_var}
\xi = - i \left( \lambda - \frac{ 1 }{ \overline \lambda } \right), \quad \bar \xi = i \left( \bar \lambda - \frac{ 1 }{ \lambda } \right).
\end{equation}
Denote
\[
\varphi( \lambda ) := - i \left( \lambda - \frac{ 1 }{ \overline \lambda } \right) =\xi.
\]
Note that
\begin{equation*}
\varphi \colon \{ \la \in \Com \ : \ | \lambda | > 1 \} \longrightarrow \{ \xi \in \Com^2 \ : \  \xi \neq 0 \}
\end{equation*}
is a bijective, smooth map. Note also that the Jacobian determinant of this transformation satisfies
\be\label{Jacobian}
\frac{ D( \xi, \bar \xi ) }{ D( \lambda, \overline \lambda ) } = 1 - \frac{ 1 }{ |\lambda|^4 } =\frac{|\la|^4-1}{|\la|^4} .
\ee
Note also that
\be\label{power_alpha}
|\xi|^\al = \abs{\lambda - \frac{ 1 }{ \overline \lambda}}^\al = \frac{|\la\bar \la-1|^\al}{|\la|^\al}.
\ee
Finally we remark the almost trivial fact (but very important in the following computations): one has
\be\label{trivial}
\frac{ \bar \lambda }{ \lambda } = - \frac{ \bar \xi }{ \xi }.
\ee
From \eqref{st_phase}, \eqref{ch_var} and \eqref{trivial}, we have
\bee
i  \widetilde S( u, \xi ,-1 ) & =& i\left( ( \xi^3 + \bar \xi^3 )\left( 1 + \frac{ 3 }{ |\xi|^2 } \right) + \frac{ 1 }{ 2 }( \xi  \bar u +  \bar \xi u) \right) \nonu\\
& =& i\left( \xi^3 + \bar \xi^3 + 3 \xi \frac{\xi}{\overline{\xi}} +3 \bar \xi \frac{\bar \xi}{\xi} \right)  + \frac{ 1 }{ 2 } \left( \left( \lambda - \frac{ 1 }{ \overline \lambda } \right) \bar u - \left( \bar \lambda - \frac{ 1 }{ \lambda } \right) u \right)  \nonu\\
&= & i\left( i\Big(  \lambda^3 - 3 \frac{\la^2}{\overline \la} + 3\frac{\la}{ \overline \la^2}   - \frac{1}{\overline \la^3} \Big) -  i\Big(  \bar \lambda^3 - 3 \frac{\bar \la^2}{\la} + 3\frac{\bar \la}{\la^2}   - \frac{1}{\la^3} \Big) \right) \\
& & + i\left(  3i \left( \lambda - \frac{ 1 }{ \overline \lambda } \right) \frac{\la}{\overline{\la}}   - 3i \left( \bar \lambda - \frac{ 1 }{ \lambda } \right)  \frac{\bar \la}{\la} \right) \nonu\\
& &  + \frac{ 1 }{ 2 } \left( \left( \lambda - \frac{ 1 }{ \overline \lambda } \right) \bar u - \left( \bar \lambda - \frac{ 1 }{ \lambda } \right) u \right)
\eee
Rearranging similar terms, we obtain
\bee
i  \widetilde S( u, \xi ,-1 ) & =& -\left( \lambda^3 - 3 \frac{\la^2}{\overline \la} + 3\frac{\la}{ \overline \la^2}   - \frac{1}{\overline \la^3}  -  \bar \lambda^3 + 3 \frac{\bar \la^2}{\la} - 3\frac{\bar \la}{\la^2}   + \frac{1}{\la^3}  \right) \\
& & - \left(  3 \frac{\la^2}{\overline{\la}}  - 3\frac{\la }{ \overline \lambda^2}    - 3 \frac{\bar \la^2}{\la}  +3 \frac{\bar \la}{\la^2}   \right) \nonu\\
& &  + \frac{ 1 }{ 2 } \left( \left( \lambda - \frac{ 1 }{ \overline \lambda } \right) \bar u - \left( \bar \lambda - \frac{ 1 }{ \lambda } \right) u \right)  \nonu  \\
& =&  - \left( \lambda^3 + \frac{ 1 }{ \lambda^3 } - \overline \lambda^3 - \frac{ 1 }{ \overline \lambda^3 } \right) + \frac{ 1 }{ 2 } \left( \left( \lambda - \frac{ 1 }{ \overline \lambda } \right) \bar u - \left( \bar \lambda - \frac{ 1 }{ \lambda } \right) u \right).
\eee
This last expression will be important for us. We denote it
\be\label{S_u}
 S( u, \lambda) :=  - \left( \lambda^3 + \frac{ 1 }{ \lambda^3 } - \overline \lambda^3 - \frac{ 1 }{ \overline \lambda^3 } \right) + \frac{ 1 }{ 2 } \left( \left( \lambda - \frac{ 1 }{ \overline \lambda } \right) \bar u - \left( \bar \lambda - \frac{ 1 }{ \lambda } \right) u \right).
\ee
Compared with $ \widetilde S( u, \xi ,-1 )$, in the new expression $ S( u, \lambda)$ \emph{the terms in $\la$ and in $\bar \la$ are decoupled}. The advantage of this fact is that in terms of $\la$ the stationary points of the oscillatory integral $I(t,u;-1)$ satisfy an algebraic equation of only one complex variable (see \eqref{asp_stationary}), while  in terms of $\xi$ the stationary points are defined by a system of two coupled algebraic equations (of two real variables).

\medskip

Moreover, note that $S( u, \lambda)$ is always a pure imaginary number, since it is the difference between a complex number and its complex conjugate.

\medskip

Thus, replacing \eqref{Jacobian}, \eqref{power_alpha} and \eqref{S_u} in \eqref{I}, we are brought to estimate the following integral
\begin{equation}
\label{new_type}
I (t,u;-1)= \iint\limits_{ \mathbb{C} \backslash B_1( 0 ) } \frac{ | \lambda \bar \lambda - 1 |^{ \alpha +i\bt } ( | \lambda |^4 - 1 ) }{ | \lambda |^{ \alpha + 4 +i\bt} } e^{ t S( u, \lambda ) } d \Re \lambda d \Im \lambda,
\end{equation}
where $B_1(0)$ is the open ball centered at zero, of radius one. We also remark that estimates of integrals of type (\ref{new_type}) have been obtained in \cite{KN, K}, see also previous works \cite{KPV_Indiana,KPV} concerning the KdV equation in one dimension. However, this time we need an improved estimate because we need a suitable gain of derivatives in order to close the iteration argument.

\subsection{Study of stationary points}

We will start by introducing the following parametrically defined sets of the complex plane  (see Fig. \ref{u_curve_figure}). Let
\begin{equation*}
\mathcal{U} := \{ u\in \Com \ : \  u = 6 ( 2 e^{ - i \varphi } + e^{ 2 i \varphi } ), \; \varphi \in [ 0, 2 \pi ) \},
\end{equation*}
and
\begin{equation*}
\mathbb{U} := \{ u\in \Com \ : \  u =  6t  (2 e^{ - i \varphi } + e^{ 2 i \varphi }), \, t\in [0,1]  \; \varphi \in [ 0, 2 \pi ) \},
\end{equation*}
be the (closed) region enclosed by the curve $ \mathcal{U} $. These sets will be essential to understand the stationary points of the phase function $S(u,\la)$.
\begin{figure}[!h]
\begin{center}
\includegraphics[width=100mm]{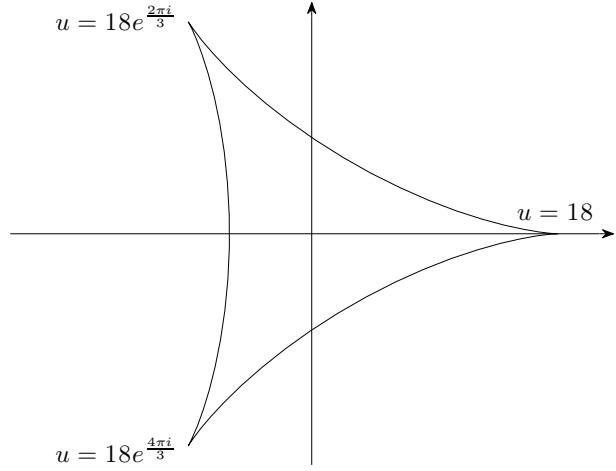}
\caption{The curve $ \mathcal{U} $ and its enclosed region $\mathbb U$ in the complex plane. Note that the $\mathcal U$ (and its interior and exterior) is invariant with respect to the transformations $z \mapsto \bar z$ and $z \mapsto ze^{2ik\pi/3}$, $k=0,1,2$.}
\label{u_curve_figure}
\end{center}
\end{figure}

With these definitions in mind, let us describe the properties of the stationary points of the function $ S( u, \lambda ) $ defined in \eqref{S_u}. These points satisfy the equation\footnote{Here the symbol $\overset{!}{=}$ means that equality to zero is satisfied for stationary points only.}
(here $S_\la$ stands for the partial derivative with respect to $\la$)
\begin{equation}
\label{asp_stationary}
S_{ \lambda } = \frac{ \bar u }{ 2 } - \frac{ u }{ 2 \lambda^2 } - 3 \lambda^2 + \frac{ 3 }{ \lambda^4 } \overset{!}{=}  0.
\end{equation}
Additionally, the degenerate stationary points obey the equation
\begin{equation}
\label{asp_degenerate}
S_{ \lambda \lambda } = \frac{ u }{ \lambda^3 } - 6 \lambda - \frac{ 12 }{ \lambda^5 } \overset{!}{=} 0.
\end{equation}
We denote $ \zeta = \lambda^2 $, and
\begin{equation*}
Q( u, \zeta ) = \frac{ \bar u }{ 2 } - \frac{ u }{ 2 \zeta } - 3 \zeta + \frac{ 3 }{ \zeta^2 }.
\end{equation*}
(Compare with \eqref{asp_stationary}.) Clearly, for each $ \zeta =\zeta(u)$, root of the function $ Q( u, \zeta ) $, there are two corresponding stationary points of $ S( u, \lambda ) $, given by $ \lambda = \pm \sqrt{ \zeta } $.
Since $Q( u, \zeta ) $ has only three roots counting multiplicity (say $\la_0^2(u)$, $\la_1^2(u)$ and $\la_2^2(u)$), in terms of the variable $\la^2$, the function $ S_{ \lambda }( u, \lambda ) $ can be represented in the following compact form
\begin{equation}
\label{asp_sprime_representation}
S_{ \lambda }( u, \lambda ) = - \frac{ 3 }{ \lambda^4 } ( \lambda^2 - \lambda_0^2( u ) ) ( \lambda^2 - \lambda_1^2( u ) ) ( \lambda^2 - \lambda_2^2( u ) ).
\end{equation}
Concerning the behavior of the roots $\la_j(u)$, defined in \eqref{asp_sprime_representation}, we have the following result, see \cite{KN} for a proof.\footnote{Note that in \cite{KN} the parametrization of the set $\mathbb U$ was slightly different, here we give a more precise one.}

\begin{lem}[Description of stationary points, \cite{KN}]\label{st_points_lemma}
\rule{1pt}{0pt}
Assume that $u\in \Com$ is a fixed parameter. Then the following are satisfied.
\begin{enumerate}
\item\label{1} If $ u = 18 e^{ \frac{ 2 \pi i k }{ 3 } } $, $ k = 0, 1, 2 $ (see the vertices of $\mathcal U$ in Fig. \ref{asp_sprime_representation}), then
\begin{equation*}
\lambda_0( u ) = \lambda_1( u ) = \lambda_2( u ) = e^{ -\frac{ \pi i k }{ 3 } },
\end{equation*}
and $ S( u, \lambda ) $ has two degenerate stationary points, corresponding to a third-order root of the function $ Q( u, \zeta ) $, $ \zeta_1 = e^{ -\frac{ 2 \pi i k }{ 3 } } $.

\item\label{2} If $ u \in \mathcal{U} $ $($i.e. $ u = 6( 2 e^{ - i \varphi } + e^{ 2 i \varphi } )$$)$ and $ u \neq 18 e^{ \frac{ 2 \pi i k }{ 3 } } $,  for $ k = 0, 1, 2 $, then \begin{equation*}
\lambda_0( u ) = \lambda_1( u ) = e^{ i \varphi / 2 }, \quad \lambda_2( u ) = e^{ - i \varphi }.
\end{equation*}

Thus $ S( u, \lambda ) $ has two degenerate stationary points, corresponding to a second-order root of the function $ Q( u, \zeta ) $, $ \zeta_1 = e^{ i \varphi } $, and two non--degenerate stationary points corresponding to a first-order root, $ \zeta_2 = e^{ - 2 i \varphi } $.

\item\label{3} If $ u \in \interior  ~ \mathbb{U} $, then
\begin{equation*}
\lambda_i( u ) = e^{ i \varphi_i }, \quad \text{and} \quad \lambda_i( u ) \neq \lambda_j( u ) \quad \text{for} \quad i \neq j.
\end{equation*}
In this case the stationary points of $ S( u, \lambda ) $ are non-degenerate and correspond to the square roots of the roots of the function $ Q( u, \zeta ) $ with absolute value equals 1.

\item\label{4} Finally, if $ u \in \mathbb{C} \backslash \mathbb{U} $, then
\begin{equation*}
\lambda_0( u ) = ( 1 + \omega ) e^{ i \varphi / 2 }, \quad \lambda_1( u ) = e^{ - i \varphi }, \quad  \lambda_2( u ) = ( 1 + \omega )^{ -1 } e^{ i \varphi / 2 },
\end{equation*}
for certain $ \varphi \in \R$ and $ \omega > 0 $.

In this case the stationary points of the function $ S( u, \lambda ) $ are non-degenerate, and correspond to the roots of the function $ Q( u, \zeta ) $ that can be expressed as $ \zeta_0 = ( 1 + \tau ) e^{ i \varphi } $, $ \zeta_1 = e^{ - 2 i \varphi } $, $ \zeta_2 = ( 1 + \tau )^{ - 1 } e^{ i \varphi } $, and $ ( 1 + \tau ) = ( 1 + \omega )^2 $.

\end{enumerate}
\end{lem}
Note that Lemma \ref{st_points_lemma} describes the behavior of stationary points with respect to the curve $\mathcal U$. Note that when estimating \eqref{new_type}, one needs to take into account that $I$ contains a multiplier that vanishes on the unit circle and grows at infinity. The neighborhoods of $|\la|=1$ and infinity being special regions, we will need to regroup the four cases above in three different ones, as expressed below.

\subsection{Idea of the proof}
Consider the integral $I=I(t,u,-1)$ in \eqref{new_type} expressed as follows:
\begin{equation}
\label{asp_more_complex_integral}
I=I(t,u,-1) = \int\limits_{ \mathbb{C} \backslash B_1( 0 ) } f( \lambda ) \exp( t S( u, \lambda ) ) d \Re \lambda \, d \Im \lambda,
\end{equation}
where $ f $ is a given function that vanishes on $ \partial B_1( 0 ) $, the boundary of the unit ball centered at the origin.

In what follows we suppress the differential $d \Re \lambda d \Im \lambda$ unless it is explicitly specified. We also anticipate that  $f$ is given by the expression
\begin{equation}\label{f}
f( \lambda ) = \frac{ | | \lambda |^2 - 1 |^{ \alpha +i\bt } ( | \lambda |^4 - 1 ) }{ | \lambda |^{ \alpha + 4 +i\bt } }.
\end{equation}
In order to estimate uniformly on $ u \in \mathbb{C} $ the large--time behavior of the integral $I$, in the following subsections we will use the following general scheme.
\label{asp_scheme_page}
\begin{enumerate}
\item Fix $\ve>0$. Consider $ D_{ \varepsilon } $ as the \emph{union of disks with radius $ \varepsilon $ and centers at the stationary points of $ S( u, \lambda ) $} in the closed set  $\mathbb{C} \backslash B_1( 0 )$.
\item Represent $ I =I(t,u,-1)$ in \eqref{new_type} as the sum of integrals over $ D_{ \varepsilon } \backslash B_1( 0 ) $ and $ \mathbb{C} \backslash ( B_1( 0 ) \cup D_{ \varepsilon } ) $:
\begin{equation}
\label{asp_int_sum}
\begin{aligned}
& I = I_{ int } + I_{ ext }, \quad \text{ where } \\
& I_{ int } = \int\limits_{ D_{ \varepsilon } \backslash B_1( 0 ) } f( \lambda ) \exp( t S( u, \lambda ) ), \qquad \hbox{ and }\\
& I_{ ext } = \int\limits_{ \mathbb{ C } \backslash ( B_1( 0 ) \cup D_{ \varepsilon } ) } f( \lambda ) \exp( t S( u, \lambda ) ) .
\end{aligned}
\end{equation}
\item Find an estimate of the form
\begin{equation*}
| I_{ int } | \lesssim  \varepsilon^{ \beta } R^{ \gamma },
\end{equation*}
(uniformly on $ u $, $ t $), for some $\bt,\ga>0$, and where $ R $ is the maximum of absolute values of the stationary points of $ S $.
\item Integrate $ I_{ ext } $ by parts using the Stokes formula \cite{} to obtain
\be\label{asp_ext_by_parts}
I_{ext} = I_1+I_2+I_3,
\ee
with
\be\label{I_1}
I_1 := \frac{ i }{ 2 t } \oint\limits_{ \partial D_{ \varepsilon } \backslash B_1( 0 ) } \frac{ f( \lambda ) \exp( t S( u, \lambda ) ) }{ S_{ \lambda }( u, \lambda ) } d \bar \lambda,
\ee
\be\label{I_2}
I_2 := - \frac{ 1 }{ t } \int\limits_{ \mathbb{C} \backslash ( B_1( 0 ) \cup D_{ \varepsilon } ) } \frac{ f_{ \lambda }( \lambda ) \exp( t S( u, \lambda ) ) }{ S_{ \lambda }( u, \lambda ) } ,
\ee
and
\be\label{I_3}
I_3 := \frac{ 1 }{ t } \int\limits_{ \mathbb{C} \backslash ( B_1( 0 ) \cup D_{ \varepsilon } ) } \frac{ f( \lambda ) \exp( t S( u, \lambda ) ) S_{ \lambda \lambda }( u, \lambda ) }{ ( S_{ \lambda }( u, \lambda ) )^2 } ,
\ee
\item For each $ I_i $ find an estimate of the form
\begin{equation*}
| I_i | \lesssim \frac{ \varepsilon^{ \delta } R^{ \kappa } }{ t }, \qquad \delta,\kappa>0,
\end{equation*}
using careful estimates on particular regions of the plane.
\item Choose $ \varepsilon $ depending on $ R $ and $ t $ so that a) $ I_{ int } $ and $ I_{ ext } $ are bounded with respect to $ R $ and b) $ I_{ int } $ and $ I_{ ext } $ decrease (as $ t \to \infty $) with the same speed.
\end{enumerate}

\subsection{Proof of the estimate for large $ t $}
We will suppose that $ t \geq e^{ - \frac{ 3 }{ 1 - \alpha } } $ (see subsection \ref{small_t_subsection} for the explanation of the choice of the constant).

First, we introduce some notation. Note that for $f(\la)$ defined in \eqref{f} we have $(f_\la =\partial_\la f )$
\[
f_\la (\la) = f_{\la,1}(\la) + f_{\la,2}(\la) + f_{\la,3}(\la),
\]
where
\be\label{f_la_1}
f_{\la,1}(\la) := \frac{ (\alpha +i\bt) | | \lambda |^2 - 1 |^{ \alpha +i\bt - 1 } \bar \lambda ( | \lambda |^4 - 1 ) }{ | \lambda |^{ \alpha  + 4 +i\bt } } ,
\ee
\be\label{f_la_2}
f_{\la,2}(\la) :=  \frac{ 2 | | \lambda |^2 - 1 |^{ \alpha +i\bt } \bar \lambda }{ | \lambda |^{ \alpha + 2+i\bt } },
\ee
and
\be\label{f_la_3}
f_{\la,3}(\la) :=  - \frac{ ( \alpha + 4+i\bt ) | | \lambda |^2 - 1 |^{ \alpha +i\bt} ( | \lambda |^4 - 1 ) \bar \lambda }{ 2 | \lambda |^{ \alpha + 6+i\bt } }.
\ee
%
In the following we consider three different cases, depending on the values of the parameter $ u $:

\ben
\item Case 1. Here $ u \in \mathbb{C} \backslash \mathbb{U} $ is such that $  | \lambda_0 | =1+w  \geq 2 $ (see Lemma \ref{st_points_lemma}, item \ref{4}).
\item Case 2. Here $ u \in \mathbb{U} $ (corresponding to items \ref{1}, \ref{2} and \ref{3} in Lemma \ref{st_points_lemma}), and finally,
\item Case 3. The region where $ u \in \mathbb{C} \backslash \mathbb{U} $ and $|\la_0| = 1 + \omega < 2 $ (see Lemma \ref{st_points_lemma}, item \ref{4}).
\een

\subsubsection{Case 1}
In this case the stationary points $ \lambda_0 $ and $ - \lambda_0 $ are separated from the unit circle, but $\pm \la_1(u)$ lies on the unit circle. Fortunately, the third root $\la_2(u)$ is uniformly outside the region of integration, since $|\la_0|= 1+w\geq 2$.

We take $ D_{ \varepsilon } $ to be the union of disks with radius $ \varepsilon $ (to be chosen later) and with centers in points $ \lambda_0 $, $ -\lambda_0 $, $ \lambda_1 $, $ -\lambda_1 $ (see Fig. \ref{D_e}). We also define
\be\label{R}
R:= |\la_0| =1+w,
\ee
and we assume that $\ve \lesssim R$.

\begin{figure}
\begin{center}
\begin{tikzpicture}[
	>=stealth',
	axis/.style={semithick,->},
	coord/.style={dashed, semithick},
	yscale = 1,
	xscale = 1]
	\newcommand{\xmin}{-4};
	\newcommand{\xmax}{4};
	\newcommand{\ymin}{-3};
	\newcommand{\ymax}{4};
	\newcommand{\ta}{3};
	\newcommand{\fsp}{0.2};
	\filldraw[color=light-gray1] (0,0) circle (2);
	\filldraw[color=light-gray2] (0,0) circle (1);
	\draw [axis] (\xmin,0) -- (\xmax,0) node [right] {$\re \la$};
	\draw [axis] (0,\ymin) -- (0,\ymax) node [below left] {$\ima \la$};
	\filldraw[color=light-gray3] (3.5,1.5) circle (0.2);
	\filldraw[color=light-gray3] (-3.5,-1.5) circle (0.2);
	\filldraw[color=light-gray3] (-0.72,0.72) circle (0.2);
	\filldraw[color=light-gray3] (0.72,-0.72) circle (0.2);
	\draw (-0.9,-0.2) node [left] {$1$};
	\draw (2.2,-0.2) node [left] {$2$};
	\draw (3.5,1.6) node [above] {$B(\la_0,\ve)$};
	\draw (-3.5,-2.4) node [above] {$B(-\la_0,\ve)$};
	\draw (-2,2.1) node [above] {$\la_2, -\la_2$};
	\draw (2,2) node [above] {$\la_1,-\la_1$};
	\draw [dashed] (-2,2) -- (0.4,0.2);
	\draw [dashed] (-2,2) -- (-0.4,-0.2);
	\draw [dashed] (2,2) -- (-0.72,0.72);
	\draw [dashed] (2,2) -- (0.72,-0.72);
	\fill (0.4,0.2)  circle[radius=1pt];
	\fill (-0.4,-0.2)  circle[radius=1pt];
	\fill (-0.72,0.72)  circle[radius=1pt];
	\fill (0.72,-0.72)  circle[radius=1pt];
	\fill (3.5,1.5)  circle[radius=1pt];
	\fill (-3.5,-1.5)  circle[radius=1pt];
\end{tikzpicture}
\end{center}
\caption{Standard setting for Case 1, with one pair of roots on the circle $|\la|=1$, another pair inside the disk $|\la|<1$ and the last pair outside the disk $|\la|< 2$. The set $D_\ve$ is the dark shadowed region.}\label{D_e}
\end{figure}
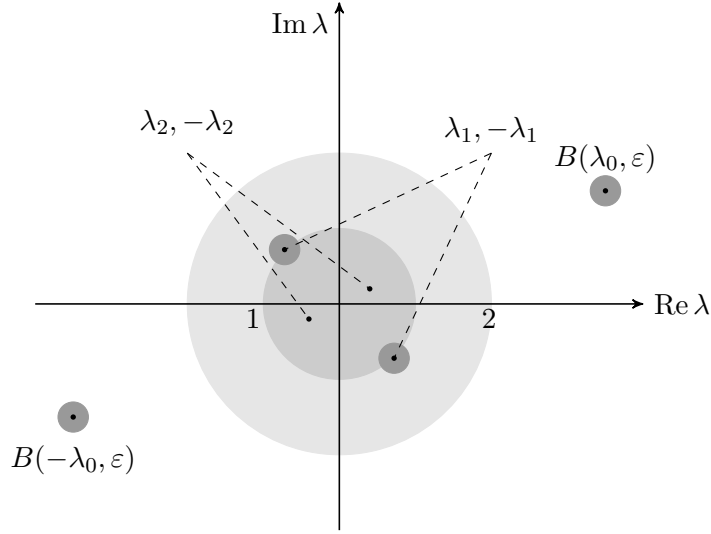

\noindent
We start by estimating $ I_{ int } $ from \eqref{asp_int_sum}. Note that for  $ \la \in D_{ \varepsilon } \backslash B_1( 0 ) $ we easily have
\[
1\leq |\la| \leq |\la_0|+\ve =R+\ve \lesssim R,
\]
and therefore the following estimates are valid:
\be\label{Est_a1}
\frac{ ( | \lambda |^2 - 1 )^{ \alpha } }{ | \lambda |^{ \alpha} } \lesssim | \lambda |^{ \alpha } \lesssim R^{ \alpha },
\ee
and
\be\label{Est_a2}
\frac{ | \lambda |^4 - 1 }{ | \lambda |^4 }  =1-\frac{1}{|\la|^4}\leq 1.
\ee
Thus, using the fact that $S(u,\la)$ in \eqref{S_u} is a pure imaginary quantity, and using \eqref{f}, \eqref{Est_a1} and \eqref{Est_a2},
\bee
| I_{ int } | & \leq &  \int_{ D_{ \varepsilon } \backslash B_1( 0 ) } |f( \lambda )| |\exp( t S( u, \lambda ) )|  \\
& \lesssim &  \int_{ D_{ \varepsilon } \backslash B_1( 0 ) }\frac{ ( | \lambda |^2 - 1 )^{ \alpha } ( | \lambda |^4 - 1 ) }{ | \lambda |^{ \alpha + 4 } }  \\
 & \lesssim & R^{ \alpha } \int_{ D_{ \varepsilon } \backslash B_1( 0 ) } d \Re \lambda \, d  \Im \lambda.
\eee
We conclude that
\be\label{Case1_I_int}
|I_{int}|   \lesssim  \varepsilon^2 R^{ \alpha }.
\ee
We pass to the estimate of $ I_{ ext } $. We will start by splitting each of the integrals $ I_2 $ and $ I_3 $ into two pieces, the first one being an integral over the set of $ \lambda $ at a distance greater than some fixed value from the stationary points, and the second one over the rest of the domain $ \mathbb{C} \backslash ( D_{ \varepsilon } \cup B_1( 0 ) ) $. More precisely, we set
\begin{equation*}
\mathbb{C} \backslash \Omega :=\{ \lambda \in \Com \colon | \lambda | \geq 2, \, | \lambda - \lambda_0 | \geq 1, \, | \lambda + \lambda_0 | \geq 1 \},
\end{equation*}
and consequently,
\bea\label{Omega_DB}
\Omega \backslash ( D_{ \varepsilon } \cup B_1( 0 ) ) & =&\{ \lambda \in \Com \ \colon \ \ve\leq | \lambda \pm  \lambda_0 | < 1 \} \nonu \\
& &  \cup \Big( \{\lambda \in \Com \ \colon \ 1 \leq | \lambda | < 2 \} \cap \{\lambda \in \Com \ \colon \,  |   \lambda \pm \lambda_1 | \geq \ve  \} \Big).
\eea

\begin{figure}
\begin{center}
\begin{tikzpicture}[
	>=stealth',
	axis/.style={semithick,->},
	coord/.style={dashed, semithick},
	yscale = 1,
	xscale = 1]
	\newcommand{\xmin}{-3.5};
	\newcommand{\xmax}{3.5};
	\newcommand{\ymin}{-3.5};
	\newcommand{\ymax}{3.5};
	\newcommand{\ta}{3};
	\newcommand{\fsp}{0.2};
	\filldraw[color=light-gray1] (0,0) circle (2);
	\filldraw[color=light-gray2] (0,0) circle (1);
	\draw [axis] (\xmin-\fsp,0) -- (\xmax,0) node [right] {$\re \la$};
	\draw [axis] (0,\ymin-\fsp) -- (0,\ymax) node [below left] {$\ima \la$};
	\filldraw[color=light-gray3] (3.5,1.5) circle (1);
	\filldraw[color=light-gray3] (-3.5,-1.5) circle (1);
	\draw (-0.9,-0.2) node [left] {$1$};
	\draw (2.2,-0.2) node [left] {$2$};
	\draw (3.5,1.5) node [above] {$\la_0$};
	\draw (-3.5,-1.5) node [above] {$-\la_0$};
	\draw (-2,2.1) node [above] {$\la_2, -\la_2$};
	\draw (2,2) node [above] {$\la_1,-\la_1$};
	\draw [dashed] (-2,2) -- (0.4,0.2);
	\draw [dashed] (-2,2) -- (-0.4,-0.2);
	\draw [dashed] (2,2) -- (-0.72,0.72);
	\draw [dashed] (2,2) -- (0.72,-0.72);
	\fill (0.4,0.2)  circle[radius=1pt];
	\fill (-0.4,-0.2)  circle[radius=1pt];
	\fill (-0.72,0.72)  circle[radius=1pt];
	\fill (0.72,-0.72)  circle[radius=1pt];
	\fill (3.5,1.5)  circle[radius=1pt];
	\fill (-3.5,-1.5)  circle[radius=1pt];
\end{tikzpicture}
\end{center}
\caption{The set $ \Com \backslash\Omega$ from Case 1 is the white open region.}\label{CminusO}
\end{figure}
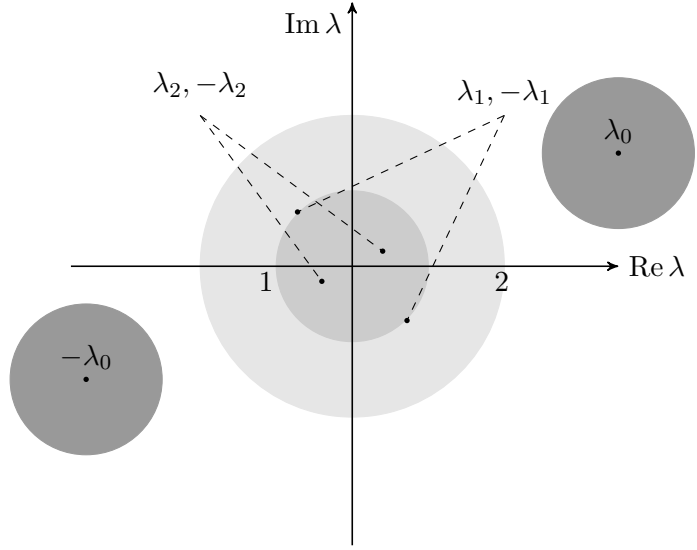
\noindent
(see Figures \ref{D_e}, \ref{CminusO} and \ref{O_De_B}) and we denote $ I_2^{ \pm } $, $ I_3^{ \pm } $ integrals with the same integrands as $ I_2$, $ I_3 $ respectively and with the domain of integration being $ \mathbb{C} \backslash \Omega $ for the $ + $ sign and $ \Omega \backslash (D_{ \varepsilon } \cup B_1( 0 ) ) $ for the $ - $ sign. More precisely, from \eqref{I_2},
\be\label{I_2jp}
 I_2^{  +} := - \frac{ 1 }{ t } \int\limits_{\mathbb{C} \backslash \Omega} \frac{ f_{ \lambda }( \lambda ) \exp( t S( u, \lambda ) ) }{ S_{ \lambda }( u, \lambda ) },
\ee
\be
 I_2^{  -} := - \frac{ 1 }{ t } \int\limits_{\Omega \backslash (D_{ \varepsilon } \cup B_1( 0 ) ) } \frac{ f_{ \lambda }( \lambda ) \exp( t S( u, \lambda ) ) }{ S_{ \lambda }( u, \lambda ) },
\ee
and from \eqref{I_3},
\be
I_3^+:= \frac{ 1 }{ t } \int\limits_{\mathbb{C} \backslash \Omega } \frac{ f( \lambda ) \exp( t S( u, \lambda ) ) S_{ \lambda \lambda }( u, \lambda ) }{ ( S_{ \lambda }( u, \lambda ) )^2 }
\ee
and
\be
I_3^- := \frac{ 1 }{ t } \int\limits_{ \Omega \backslash (D_{ \varepsilon } \cup B_1( 0 ) ) } \frac{ f( \lambda ) \exp( t S( u, \lambda ) ) S_{ \lambda \lambda }( u, \lambda ) }{ ( S_{ \lambda }( u, \lambda ) )^2 } .
\ee

We first treat integrals $ I_2^{  + } $, $ I_3^{ + } $ (our goal is to obtain an estimate for these integrals independent of $ u $). For that we use the following estimates valid on $ \mathbb{C} \backslash \Omega $ (see Fig. \ref{CminusO}):
\bea\label{star1}
\left| \frac{ | \lambda |^4 - 1 }{ | \lambda |^4 } \right| \lesssim 1, \quad \left| \frac{ | \lambda |^2 - 1 }{ | \lambda | } \right| \lesssim | \lambda |, \nonu\\
\frac{ | ( \lambda - \lambda_1 ) ( \lambda + \lambda_1 ) ( \lambda - \lambda_2 ) ( \lambda + \lambda_2 ) | }{ | \lambda |^4 } \gtrsim 1.
\eea
From these estimates and \eqref{f_la_1}-\eqref{f_la_3} we get
\[
|f_{\la,1}| \lesssim (1+|\bt|) |\la|^{\al-1},
\]
\[
|f_{\la, 2}| \lesssim |\la|^{\al-1},
\]
and
\[
|f_{\la, 3}| \lesssim (1+|\bt|)|\la|^{\al-1}.
\]
We also conclude from \eqref{asp_sprime_representation} and \eqref{star1} that
\be\label{S_la_lower}
|S_\la(\cdot,\la)| \gtrsim |\la^2 -\la_0^2| = |\la + \la_0| |\la -\la_0|.
\ee
Thus from \eqref{I_2jp} we obtain the following estimate, valid for each $j=1,2,3$:
\begin{equation*}
| I_2^{ + } | \lesssim \frac{ 1 }{ t }(1+\bt) \int\limits_{ \mathbb{C} \backslash \Omega } \frac{ 1 }{ | \lambda - \lambda_0 | | \lambda + \lambda_0 | | \lambda |^{ 1 - \alpha } }=: (1+\bt)J.
\end{equation*}
Now we compute $J$. We split $ \mathbb{C} \backslash \Omega $ into four regions:
\[
D_1 := \{ \lambda \in \mathbb{C} \backslash \Omega ~ \colon | \lambda - \lambda_0 | \leq | \lambda |, | \lambda - \lambda_0 | \leq | \lambda + \lambda_0 | \},
\]
\[
D_2 := \{ \lambda \in \mathbb{C} \backslash \Omega ~ \colon | \lambda - \lambda_0 | \geq | \lambda |, | \lambda - \lambda_0 | \leq | \lambda + \lambda_0 | \},
\]
\[
D_3 := \{ \lambda \in \mathbb{C} \backslash \Omega ~ \colon | \lambda + \lambda_0 | \geq | \lambda |, | \lambda - \lambda_0 | \geq | \lambda + \lambda_0 | \},
\]
and
\[
D_4 := \{ \lambda \in \mathbb{C} \backslash \Omega ~ \colon | \lambda + \lambda_0 | \leq | \lambda |, | \lambda - \lambda_0 | \geq | \lambda + \lambda_0 | \}.
\]
We split the integral $ J $ accordingly
\begin{equation*}
J = \sum_{ j = 1 }^{ 4 } \frac{ 1 }{ t } \int\limits_{ D_j } \frac{ 1 }{ | \lambda - \lambda_0 | | \lambda + \lambda_0 | | \lambda |^{ 1 - \alpha } } =: \sum_{ j = 1 }^{ 4 } J_j.
\end{equation*}
Now we are ready to give uniform (w.r.t. $ u $) estimates of $ J_j $:
\begin{align*}
& | J_1 | \leq \frac1t \int\limits_{ D_1 } \frac{ 1 }{ | \lambda - \lambda_0 |^{ 3 - \alpha } } \leq \frac{ 1 }{ t } \int\limits_{ | \lambda | \geq 1 } \frac{ 1 }{ | \lambda |^{ 3 - \alpha } } \lesssim \frac{ 1 }{ t }, \\
& | J_2 | \leq \frac{ 1 }{ t } \int\limits_{ D_2 } \frac{ 1 }{ | \lambda |^{ 3 - \alpha } } \leq \frac{ 1 }{ t } \int\limits_{ | \lambda | \geq 2 } \frac{ 1 }{ | \lambda |^{ 3 - \alpha } } \lesssim \frac{ 1 }{ t }, \\
& | J_3 | \leq \frac{ 1 }{ t } \int\limits_{ D_3 } \frac{1}{ | \lambda |^{ 3 - \alpha } } \leq \frac{ 1 }{ t } \int\limits_{ | \lambda | \geq 2 } \frac{ 1 }{ | \lambda |^{ 3 - \alpha } } \lesssim \frac{ 1 }{ t }, \\
& | J_4 | \leq \frac{ 1 }{ t } \int\limits_{ D_4 } \frac{ 1 }{ | \lambda + \lambda_0 |^{ 3 - \alpha } } \leq \frac{ 1 }{ t } \int\limits_{ | \lambda | \geq 1 } \frac{ 1 }{ | \lambda |^{ 3 - \alpha } } \lesssim \frac{ 1 }{ t }.
\end{align*}
Note that these estimates are valid only if $ \alpha < 1 $. We conclude that
\be\label{Est_I2j_p}
| I_2^{ + } | \lesssim (1+|\bt|)\frac1t.
\ee

The estimate for $ I_{ 3 }^{ + } $ is carried out similarly. First of all, note that from \eqref{asp_sprime_representation}
\bea
S_{ \lambda \lambda } & =&  \partial_\la \Bigg[\frac{- 3 }{ \lambda^4 } ( \lambda^2 - \lambda_0^2 ) ( \lambda^2 - \lambda_1^2) ( \lambda^2 - \lambda_2^2) \Bigg]  \nonu\\
& =&  \sum_{ j = 0, 1, 2 } \left( \frac{ S_{ \lambda } }{ ( \lambda - \lambda_j ) } + \frac{ S_{ \lambda } }{ ( \lambda + \lambda_j ) } \right) - 4 \frac{ S_{ \lambda } }{ \lambda } \label{S_ll} \\
 &=:&  S_{ \lambda \lambda }^{ 1 } + \ldots + S_{ \lambda \lambda }^{ 7 }. \nonu
\eea
Further, on $ \mathbb{C} \backslash \Omega $, and using \eqref{S_la_lower},
\begin{gather*}
\left| \frac{ S_{ \lambda \lambda }^{ 1 } }{ ( S_{ \lambda } )^2 } \right| \lesssim \frac{ 1 }{ | \lambda - \lambda_0 |^2 | \lambda + \lambda_0 | }, \quad \left| \frac{ S_{ \lambda \lambda }^{ 2 } }{ ( S_{ \lambda } )^2 } \right| \lesssim \frac{ 1 }{ | \lambda - \lambda_0 | | \lambda + \lambda_0 |^2 }, \\
\left| \frac{ S_{ \lambda \lambda }^{ j } }{ ( S_{ \lambda } )^2 } \right| \lesssim \frac{ 1 }{ | \lambda - \lambda_0 | | \lambda + \lambda_0 | | \lambda | }, \quad j = 3, \ldots, 7.
\end{gather*}
From now the estimates are obtained similarly to the way they were obtained for $ I_2^{  + } $; more precisely, by splitting the domain of integration into the subdomains $ D_1, \ldots, D_4 $. We conclude
\be\label{Est_I3_p}
| I_3^{ + } | \lesssim \frac{ 1 }{ t } \int\limits_{\mathbb{C} \backslash \Omega } \frac{ |f( \lambda )| |S_{ \lambda \lambda }( u, \lambda )| }{ |S_{ \lambda }( u, \lambda ) |^2 }  \lesssim \frac 1t \int_{|\la|\geq 1} \frac{1}{|\la|^{3-\al}} \lesssim  \frac1t.
\ee

\medskip

\begin{figure}
\begin{center}
\begin{tikzpicture}[
	>=stealth',
	axis/.style={dashed,->},
	coord/.style={dashed, semithick},
	yscale = 1,
	xscale = 1]
	\newcommand{\xmin}{-3.5};
	\newcommand{\xmax}{3.5};
	\newcommand{\ymin}{-3.5};
	\newcommand{\ymax}{3.5};
	\newcommand{\ta}{3};
	\newcommand{\fsp}{0.2};
	\filldraw[color=light-gray2] (0,0) circle (2);
	\filldraw[color=light-gray1] (-0.72,0.72) circle (0.2);
	\filldraw[color=light-gray1] (0.72,-0.72) circle (0.2);
	\draw (-0.72,0.72) circle (0.2);
	\draw (0.72,-0.72) circle (0.2);
	\filldraw[color=light-gray1] (0,0) circle (1);
	\draw [axis] (\xmin-\fsp,0) -- (\xmax,0) node [right] {$\re \la$};
	\draw [axis] (0,\ymin-\fsp) -- (0,\ymax) node [below left] {$\ima \la$};
	\filldraw[color=light-gray2] (3.5,1.5) circle (1);
	\filldraw[color=light-gray2] (-3.5,-1.5) circle (1);
	\draw (-0.9,-0.2) node [left] {$1$};
	\draw (2.2,-0.2) node [left] {$2$};
	\draw (2,2) node [above] {$\la_1,-\la_1$};
	\draw [dashed] (2,2) -- (-0.72,0.72);
	\draw [dashed] (2,2) -- (0.72,-0.72);
	\filldraw[color=light-gray1] (3.5,1.5) circle (0.2);
	\filldraw[color=light-gray1] (-3.5,-1.5) circle (0.2);
	\draw (3.5,1.7) node [above] {$\la_0$};
	\draw (-3.5,-1.35) node [above] {$-\la_0$};
	\fill (0.4,0.2)  circle[radius=1pt];
	\fill (-0.4,-0.2)  circle[radius=1pt];
	\fill (-0.72,0.72)  circle[radius=1pt];
	\fill (0.72,-0.72)  circle[radius=1pt];
	\fill (3.5,1.5)  circle[radius=1pt];
	\fill (-3.5,-1.5)  circle[radius=1pt];
	\draw (-3.5,-1.5) circle (0.2);
	\draw (3.5,1.5) circle (0.2);
\end{tikzpicture}
\end{center}
\caption{The dark region is $ \Omega \backslash ( D_{ \varepsilon } \cup B_1( 0 ) )$. The curves composing $ \partial D_{ \varepsilon } \backslash B_1( 0 ) $ in $I_1$ (see \eqref{I_1}) are the darkest continuous lines.}\label{O_De_B}
\end{figure}
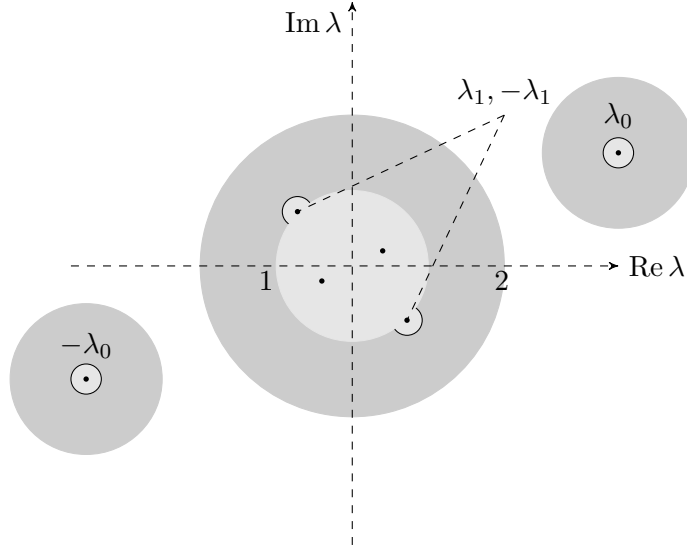
\noindent

%

Now we pass to estimating integrals $ I_1 $, $ I_2^{  - } $, and $ I_3^{ - } $. The reference domain now is $ \Omega \backslash ( D_{ \varepsilon } \cup B_1( 0 ) )$, described in \eqref{Omega_DB} and Fig. \ref{O_De_B}. First of all, let us denote
\[
\lambda^{ ( 1 ) } = \lambda_0, \quad \lambda^{ ( 2 ) } = - \lambda_0 , \quad \lambda^{ ( 3 ) } = \lambda_1,\quad  \lambda^{ ( 4 ) } = - \lambda_1, \quad  \lambda^{ ( 5 ) } =  \lambda_2,\quad  \lambda^{ ( 6) } = - \lambda_2.
\]
We split the integral $ I_1 $ into the sum of the following integrals
\bea
I_1 & =&  \frac{ i }{ 2 t } \sum_{ k = 1 }^{ 4 } \oint\limits_{ \partial B_{ \varepsilon }( \lambda^{ ( k ) } ) \backslash B_1( 0 ) } \frac{ f( \lambda ) e^{ t S( u, \lambda ) } }{ S_{ \lambda }( \lambda ) } d \bar \lambda \nonu \\
&  =:&  \sum_{ k= 1 }^{ 4 } I_1^{ k}. \label{I_1k}
\eea
Further, we split $ \Omega \backslash ( D_{ \varepsilon } \cup B_1( 0 ) ) $ into four domains ($k = 1, \ldots, 4$)
\be\label{O_k}
\Omega^{ (k) } := \big\{ \lambda \in \Omega \backslash ( B_1( 0 ) \cup D_{ \varepsilon } ) ~ \colon | \lambda - \lambda^{ ( k ) } | \leq | \lambda - \lambda^{ ( \ell ) } |, ~ \ell \neq k, ~ \ell = 1, \ldots, 4 \big\}.
\ee
Note additionally the essential inclusion, valid for each $k=1,\ldots, 4$,
\be\label{Neat_Inclusion}
\partial B_{ \varepsilon }( \lambda^{ ( k ) } ) \backslash B_1( 0 ) \subset \Omega^{(k)}.
\ee
Finally, we put
\be\label{I_2_jmk}
I_2^{ -,k } := \frac{ 1 }{ t } \int\limits_{ \Omega^{ ( k ) } } \frac{ f_{ \la }( \lambda ) e^{ i t S( u, \lambda ) } }{ S_{ \lambda }( u, \lambda ) } ,
\ee
and
\be\label{I_3_mk}
I_3^{-,k} := \frac{ 1 }{ t } \int\limits_{ \Omega^{ ( k ) } } \frac{ f( \lambda ) S_{ \lambda \lambda }( u, \lambda ) e^{ i t S( u, \lambda ) } }{ ( S_{ \lambda }( u, \lambda ) )^2 } .
\ee
In order to estimates these integrals and each $I_1^j$, we follow the approach that we described now. For every integral over domain $ \Omega^{ ( k ) } $, we use the following polar coordinates at the stationary point $\la^{(k)}$:
\[
\lambda = \lambda^{ (k) } + \rho e^{ i \varphi }, \quad \varphi \in [0,2\pi),
\]
and where $\ve\leq \rho \leq \rho_0(\varphi)$. Note that $\rho_0$ is always bounded by a fixed constant $\ve_0$, that can be chosen e.g. equals 4, uniformly on $\varphi$.

\medskip

Now we start computing. Fix $\rho $ as previously mentioned.

\medskip

\noindent
1. On $ \partial B_{ \rho }(\lambda^{ (1) })  \cap \Omega^{ (1) } $ we have $1\leq |\la| \leq |\la_0| +1$, so we use the following estimates:
\[
\left| \frac{ | \lambda |^4 - 1 }{ | \lambda |^4 } \right| \lesssim 1, \quad \left| \frac{ | \lambda |^2 - 1 }{ | \lambda | } \right| \lesssim R,
\]
(recall that $R= |\la_0|$, defined in \eqref{R}). On the other hand, using \eqref{asp_sprime_representation},
\[
| S_{ \lambda } | \sim \frac{ 1 }{ |\lambda|^4 }  |\lambda - \lambda^{(1)}|  |\lambda - \lambda^{(2)}| |\lambda - \lambda^{(3)}|  |\lambda - \lambda^{(4)}|   |\lambda - \lambda^{(5)}||\lambda - \lambda^{(6)}| \gtrsim \rho R,
\]
and from \eqref{S_ll},
\[
\frac{ S_{ \lambda \lambda } }{ ( S_{ \lambda } )^2 } =  \frac1{S_\la}  \Big[ \sum_{ j = 1}^6  \frac{ 1 }{  \lambda - \lambda^{(j)}  }   - 4 \frac{ 1}{ \lambda } \Big],
\]
so that
\[
\left| \frac{ S_{ \lambda \lambda } }{ ( S_{ \lambda } )^2 } \right| \lesssim \frac{ 1 }{ \rho R^2 } + \frac{ 1 }{ \rho^2 R }.
\]
Using these observations we are able to obtain the following estimate for $I_1^1$ in \eqref{I_1k}:
\bee
| I_{ 1 }^{ 1 } | & \lesssim &   \frac{ 1 }{ t } \oint\limits_{ \partial B_{ \varepsilon }( \lambda^{ ( 1 ) } ) \backslash B_1( 0 ) } \frac{ |f( \lambda ) | } { |S_{ \lambda }( \lambda )| } d \bar \lambda \\
& \lesssim & \frac{ 1 }{ t } \frac{ R^{ \alpha } \times \varepsilon }{ \varepsilon R } = \frac{ 1 }{ t } R^{ \alpha - 1 }.
\eee
On the other hand, for \eqref{I_2_jmk}-\eqref{I_3_mk},
\bee
| I_2^{ -, 1 } | & \lesssim & \frac{ 1 }{ t } \int\limits_{ \Omega^{ ( 1 ) } } \frac{ |f_{ \la }( \lambda )| }{ |S_{ \lambda }( u, \lambda )| }\\
& \lesssim &  \frac{ 1 }{ t }(1+|\bt|) \int\limits_{ \varepsilon }^{ \varepsilon_0 } \frac{ R^{ \alpha - 1 } \rho \, d \rho }{ \rho R } \sim \frac{ 1 }{ t } (1+|\bt|)R^{ \alpha - 2 }, \quad j = 1, 2, 3,
\eee
and finally,
\bee
| I_3^{-,1} | & \lesssim & \frac{ 1 }{ t } \int\limits_{ \Omega^{ ( 1 ) } } \frac{ |f( \lambda )||S_{ \lambda \lambda }( u, \lambda )| } { | S_{ \lambda }( u, \lambda ) |^2 } \\
& \lesssim & \frac{ 1 }{ t } \int\limits_{ \varepsilon }^{ \varepsilon_0 } \frac{ R^{ \alpha } \rho d \rho }{ \rho R^2 } + \frac{ 1 }{ t } \int\limits_{ \varepsilon }^{ \varepsilon_0 } \frac{ R^{ \alpha } \rho d \rho }{ \rho^2 R } \\
& \lesssim & \frac{ 1 }{ t } R^{ \alpha - 1 } |\ln \varepsilon|.
\eee
(Recall that $ \varepsilon_0 $ is a certain fixed independent of $ u $, $ t $, $ \varepsilon $, and $ R $.)

\medskip

\noindent
2. On $ \partial B_{ \rho }(\lambda^{ (2) }) \cap \Omega^{ (2) } $, computations are perfectly symmetric with respect to the first case. We obtain the following estimates
\be
| I_{ 1 }^{ 2 } | \lesssim \frac{ 1 }{ t } R^{ \alpha - 1 }, \quad | I_2^{ -, 2 } | \lesssim \frac{ 1 }{ t }(1+\bt) R^{ \alpha - 2 }, \quad j = 1, 2, 3, \quad \hbox{ and } \quad  | I_3^{-,2} | \lesssim \frac{ 1 }{ t } R^{ \alpha - 1 } |\ln \varepsilon|.
\ee
\medskip

\noindent
3. Now we deal with the more involved case $ \partial B_{ \rho }(\lambda^{ (3) }) \cap \Omega^{ (3) } $. In this region we have $|\la^{(3)}| =1$ and $1\leq |\la|\leq 2 $. Therefore, from the decomposition $\lambda = \lambda^{ (3) } + \rho e^{ i \varphi }$ we shall use the following estimates
\[
 | | \lambda |^2 - 1 | \lesssim \rho, \quad | | \lambda |^4 - 1 | \lesssim \rho, \quad |f(\la)| \lesssim \rho^{\al+1}.
\]
Similarly,
\[
| S_{ \lambda } |  \sim \frac{ 1 }{ |\lambda|^4 }  |\lambda - \lambda^{(1)}|  |\lambda - \lambda^{(2)}| |\lambda - \lambda^{(3)}|  |\lambda - \lambda^{(4)}|   |\lambda - \lambda^{(5)}||\lambda - \lambda^{(6)}|  \gtrsim   \rho R^2.
\]
Finally, since
\[
 \frac{ S_{ \lambda \lambda } }{ ( S_{ \lambda } )^2 } =  \frac1{S_\la}  \Big[ \sum_{ j = 1}^6  \frac{ 1 }{  \lambda - \lambda^{(j)}  }   - 4 \frac{ 1}{ \lambda } \Big],
\]
we obtain
\[
\left| \frac{ S_{ \lambda \lambda } }{ ( S_{ \lambda } )^2 } \right| \lesssim \frac{ 1 }{ \rho R^2 } \Big( 1+\frac1\rho+ \frac{1}{R }\Big)  \lesssim  \frac{ 1 }{ \rho R^2 } + \frac{ 1 }{ \rho^2 R^2 } + \frac{ 1 }{ \rho R^3 }.
\]
Using these observations we are able to obtain the following estimates:
\begin{gather*}
| I_{ 1 }^{ 3 } | \lesssim \frac{ 1 }{ t } \frac{ \varepsilon \cdot \varepsilon^{ \alpha + 1 } }{ \varepsilon R^2 } = \frac{ 1 }{ t } \frac{ \varepsilon^{ \alpha +1 } }{ R^{ 2 } }, \\
| I_2^{ j, 3 } | \lesssim \frac{ 1 }{ t }(1+\bt) \int\limits_{ \varepsilon }^{ \varepsilon_0 } \frac{ \rho^{ \alpha } \rho d \rho }{ \rho R^2 } \sim \frac{ (1+|\bt|) }{ t R^{ 2 } }, \quad j = 1, 2, 3, \\
| I_3^3 | \lesssim \frac{ 1 }{ t } \int\limits_{ \varepsilon }^{ \varepsilon_0 } \rho^{\al+1} \Big[ \frac{ 1 }{ \rho R^2 } + \frac{ 1 }{ \rho^2 R^2 } + \frac{ 1 }{ \rho R^3 } \Big]  \rho \, d\rho \lesssim  \frac{ 1 }{ t R^{ 2 } }.
\end{gather*}

\medskip

\noindent
4. Finally, the case for $ \partial B_{ \rho } (\lambda^{ (4) })\cap \Omega^{ (4) } $ is very similar to the previous one. It is not difficult to see that the following estimates can be obtained
\begin{equation*}
| I_{ 1 }^{ 4 } | \lesssim \frac{ 1 }{ t } \frac{ \varepsilon^{ \alpha +1 } }{ R^{ 2 } }, \quad | I_2^{  -, 4 } | \lesssim \frac{ (1+|\bt|) }{ t R^{ 2 } }, \quad j = 1, 2, 3, \quad | I_3^{-,4} | \lesssim \frac{ 1 }{ t R^{ 2 } }.
\end{equation*}

\medskip

Adding the four previous estimates, we conclude that
\be\label{Case1_I_1}
|I_1| \leq  \sum_{j=1}^4  |I_1^j| \lesssim  \frac1{Rt} \Big[  R^\al + \frac{\ve^{\al+1}}{R}\Big],
\ee
and for $k=1,\ldots, 4$,
\be\label{Case1_I_2}
|I_2^{j,-,k}| \leq  \sum_{j=1}^4  |I_1^j| \lesssim  \frac{(1+|\bt|)}{R^2t} (1+ R^\al ),
\ee
and
\be\label{Case1_I_3}
|I_3^k| \leq  \sum_{j=1}^4  |I_1^j| \lesssim  \frac1{Rt} ( 1+ R^\al |\ln \ve| ).
\ee

\medskip

Collecting all the previous inequalities, including \eqref{Case1_I_int}, \eqref{Est_I2j_p}, \eqref{Est_I3_p} and \eqref{Case1_I_1}-\eqref{Case1_I_3}, we get the global estimate (recall that $R=|\la_0|>2$ and $t$ is w.l.o.g. assumed large)
\bee
|I| & \lesssim & \ve^2 R^\al + \frac {(1+|\bt|)}t \Big[ 1+ R^{\al-1} + \frac{\ve^{\al+1}}{R^2} +\frac 1{R^2} + R^{\al-2} +\frac1R + R^{\al-1} |\ln\ve|  \Big] \\
& \lesssim & \ve^2 R^\al + \frac {(1+|\bt|)}t \Big[ 1 +  \frac{\ve^{\al+1}}{R^2} + R^{\al-1} |\ln\ve|\Big].
\eee
By choosing $ \varepsilon := \min \big\{ \frac{\delta_0}{\sqrt{R}}, \frac{ 1 }{ \sqrt{ t R } } \big\} $, with $\delta_0>0$ small, we obtain that, if $ \alpha < 1 $, the integral $I$ is uniformly bounded with respect to $ R $ and decreases as
\[
|I| \lesssim \frac{R^{\al-1}}{tR} + \frac {(1+|\bt|)}t \Big(1 +R^{\al-1}  |\ln  tR| \Big) \lesssim \frac {(1+|\bt|)|\ln t|}t.
\]


\subsubsection{Case 2: $ u \in \mathbb{U} $}

Recall that in this case \emph{all the stationary points lie on the unit circle} (see Lemma \ref{st_points_lemma}, items 1, 2 and 3, and Fig. \ref{case_U}).

As we mentioned earlier,  we take $ D_{ \varepsilon } $ to be the union of disks with radius $ \varepsilon $ and with centers in the stationary points.
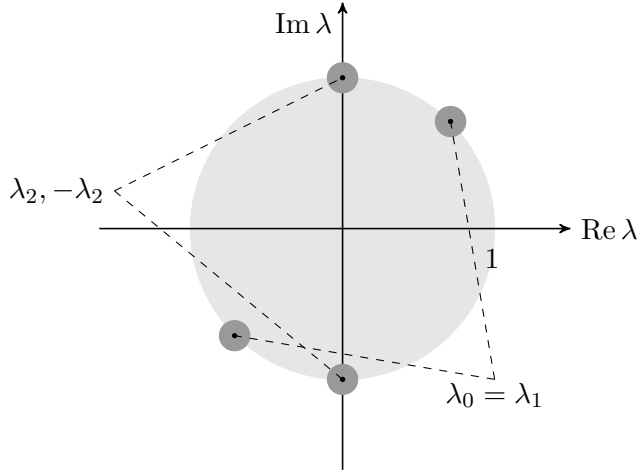
\begin{figure}
\begin{center}
\begin{tikzpicture}[
	>=stealth',
	axis/.style={semithick,->},
	coord/.style={dashed, semithick},
	yscale = 1,
	xscale = 1]
	\newcommand{\xmin}{-3};
	\newcommand{\xmax}{3};
	\newcommand{\ymin}{-3};
	\newcommand{\ymax}{3};
	\newcommand{\ta}{3};
	\newcommand{\fsp}{0.2};
	\filldraw[color=light-gray1] (0,0) circle (2);
	\draw [axis] (\xmin-\fsp,0) -- (\xmax,0) node [right] {$\re \la$};
	\draw [axis] (0,\ymin-\fsp) -- (0,\ymax) node [below left] {$\ima \la$};
	\filldraw[color=light-gray3] (1.42,1.42) circle (0.2);
	\filldraw[color=light-gray3] (-1.42,-1.42) circle (0.2);
	\filldraw[color=light-gray3] (0,2) circle (0.2);
	\filldraw[color=light-gray3] (0,-2) circle (0.2);
	\draw (2.2,-0.4) node [left] {$1$};
	\draw (-3,0.5) node [left] {$\la_2,-\la_2$};
	\draw (2,-2.5) node [above] {$\la_0=\la_1$};
	\draw [dashed] (-3,0.5) -- (0,2);
	\draw [dashed] (-3,0.5) -- (0,-2);
	\draw [dashed] (2,-2) -- (1.42,1.42);
	\draw [dashed] (2,-2) -- (-1.42,-1.42);
	\fill (0,-2)  circle[radius=1pt];
	\fill (0,2)  circle[radius=1pt];
	\fill (1.42,1.42)  circle[radius=1pt];
	\fill (-1.42,-1.42)  circle[radius=1pt];
\end{tikzpicture}
\end{center}
\caption{Setting for Case 2 (in fig., Lemma \ref{st_points_lemma}, item \ref{2}). $D_\ve$ is the dark shadowed region.}\label{case_U}
\end{figure}
\noindent
We start by estimating $ I_{ int } $ in \eqref{asp_int_sum}. First of all, note that on $ D_{ \varepsilon } \backslash B_1( 0 ) $ the following estimates are valid:
\begin{equation*}
( | \lambda |^2 - 1 )^{ \alpha } \lesssim \varepsilon^{ \alpha }, \quad | \lambda | \sim 1, \quad | \lambda |^4 - 1 \lesssim \varepsilon.
\end{equation*}
Thus, using \eqref{f},
\[
|f(\la)| \lesssim \frac{ | | \lambda |^2 - 1 |^{ \alpha} ( | \lambda |^4 - 1 ) }{ | \lambda |^{ \alpha + 4} } \lesssim  \ve^{\al +1},
\]
and
\be\label{I_int_case2}
 |I_{ int } | \lesssim   \int\limits_{ D_{ \varepsilon } \backslash B_1( 0 ) } |f( \lambda )| \lesssim  \varepsilon^{ \alpha + 3 }.
\ee

\medskip

We pass to the estimate of $ I_{ ext } $ introduced in \eqref{asp_ext_by_parts}.  Recall that $I_{ext}$ is composed of three different parts, $I_1$, $I_2$ and $I_3$, see \eqref{I_1}, \eqref{I_2} and \eqref{I_3}.  

\medskip

As in the previous case, we will start by splitting each of the integrals $ I_2^j $ and $ I_3 $ into two additional pieces, the first one being an integral over the set of $ \lambda $ at a distance greater than some fixed value from the unit circle (containing all the stationary points), and the second one over the rest of the domain included in $ \mathbb{C} \backslash ( D_{ \varepsilon } \cup B_1( 0 ) ) $. More precisely, in this case we set

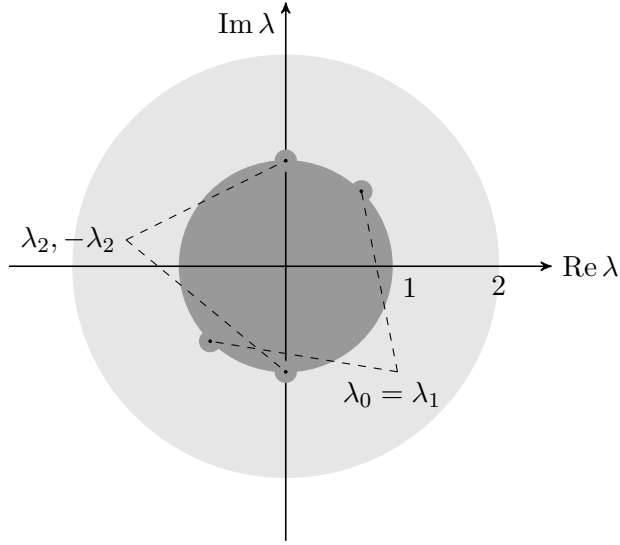
\begin{figure}
\begin{center}
\begin{tikzpicture}[
	>=stealth',
	axis/.style={semithick,->},
	coord/.style={dashed, semithick},
	yscale = 0.7,
	xscale = 0.7]
	\newcommand{\xmin}{-5};
	\newcommand{\xmax}{5};
	\newcommand{\ymin}{-5};
	\newcommand{\ymax}{5};
	\newcommand{\ta}{3};
	\newcommand{\fsp}{0.2};
	\filldraw[color=light-gray1] (0,0) circle (4);
	\filldraw[color=light-gray3] (0,0) circle (2);
	\draw [axis] (\xmin-\fsp,0) -- (\xmax,0) node [right] {$\re \la$};
	\draw [axis] (0,\ymin-\fsp) -- (0,\ymax) node [below left] {$\ima \la$};
	\filldraw[color=light-gray3] (1.42,1.42) circle (0.2);
	\filldraw[color=light-gray3] (-1.42,-1.42) circle (0.2);
	\filldraw[color=light-gray3] (0,2) circle (0.2);
	\filldraw[color=light-gray3] (0,-2) circle (0.2);
	\draw (2,-0.4) node [right] {$1$};
	\draw (4,0) node [below] {$2$};
	\draw (2,-2) node [below] {$\la_0= \la_1$};
	\draw (-3,0.5) node [left] {$\la_2,-\la_2$};
	\draw [dashed] (-3,0.5) -- (0,2);
	\draw [dashed] (-3,0.5) -- (0,-2);
	\draw [dashed] (2.1,-2) -- (1.42,1.42);
	\draw [dashed] (2.1,-2) -- (-1.42,-1.42);
	\fill (0,2)  circle[radius=1pt];
	\fill (0,-2)  circle[radius=1pt];
	\fill (1.42,1.42)  circle[radius=1pt];
	\fill (-1.42,-1.42)  circle[radius=1pt];
\end{tikzpicture}
\end{center}
\caption{Case 2. The white region is $\Com\backslash \Omega$. The light shadowed region (without small disks of radius $\ve$ around the stationary points), is $\Omega \backslash (D_{ \varepsilon } \cup B_1( 0 ) )$.}\label{case_U_1}
\end{figure}
\noindent

\begin{equation*}
\mathbb{C} \backslash \Omega: =\{ \lambda \colon | \lambda | \geq 2 \},
\end{equation*}
and we denote as $ I_2^{  \pm } $, $ I_3^{ \pm } $, the integrals with same integrands as $ I_2^{ j} $ and $ I_3 $, respectively, and with domain of integration being $ \mathbb{C} \backslash \Omega $ for the plus sign, and $ \Omega \backslash (D_{ \varepsilon } \cup B_1( 0 ) ) $ for the minus sign (see Fig. \ref{case_U_1} and Case 1 for a similar splitting).

\medskip

We first treat integrals $ I_2^{ j, + } $, $ I_3^{ + } $. For that we use the following estimates valid on $ \mathbb{C} \backslash \Omega $:
\be\label{Auxiliary_case2}
\left| \frac{ | \lambda |^4 - 1 }{ | \lambda |^4 } \right| \lesssim 1, \quad \left| \frac{ | \lambda |^2 - 1 }{ | \lambda | } \right| \lesssim | \lambda |,
\ee
and from \eqref{asp_sprime_representation},
\[
| S_{ \lambda } | \gtrsim | \lambda |^2.
\]
From the decomposition (see previous case)
\[
S_{ \lambda \lambda }(u,\la) =  \sum_{ j = 0, 1, 2 } \left( \frac{ S_{ \lambda }(u,\la) }{ ( \lambda - \lambda_j ) } + \frac{ S_{ \lambda } (u,\la)}{ ( \lambda + \lambda_j ) } \right) - 4 \frac{ S_{ \lambda }(u,\la) }{ \lambda },
\]
we also have
\[
  \left| \frac{ S_{ \lambda \lambda }(u,\la) }{ ( S_{ \lambda }(u,\la) )^2 } \right| \lesssim \frac{ 1 }{ | \lambda |^3 }.
\]
Finally, using \eqref{f} and \eqref{f_la_1}-\eqref{f_la_3}, and estimates \eqref{Auxiliary_case2}, we have that
\[
|f(\la)| \lesssim |\la|^\al , \qquad |f_{\la,1}(\la)| +|f_{\la,2}(\la)|+|f_{\la,3}(\la)| \lesssim (1+ |\bt |) |\la|^{\al-1}.
\]
Thus we obtain the following estimates (recall that $\al<1$)
\bea
 | I_2^{  + } | & \lesssim &   \frac{ 1 }{ t } \int_{\mathbb{C} \backslash \Omega } \frac{ |f_{ \lambda}( \lambda )|  }{ |S_{ \lambda }( u, \lambda )| } \lesssim \frac{ 1 }{ t } \int\limits_{ \mathbb{C} \backslash \Omega } \frac{ 1 }{ | \lambda |^{ 3 - \alpha } } \lesssim \frac{ 1 }{ t }, \label{I_2_jp_case2} \\
 | I_3^+ | & \lesssim & \frac{ 1 }{ t } \int\limits_{\mathbb{C} \backslash \Omega } \frac{| f( \lambda )| |S_{ \lambda \lambda }( u, \lambda )| }{ | S_{ \lambda }( u, \lambda ) |^2 } \lesssim \frac{ 1 }{ t } \int\limits_{ \mathbb{C} \backslash \Omega } \frac{ 1 }{ | \lambda |^{ 3 - \alpha } } \lesssim \frac{ 1 }{ t }.\label{I_3_p_case2}
\eea
\medskip

Now we pass to estimating integrals $ I_1 $, $ I_2^{  - } $, $ I_3^{ - } $. As in the previous case, we denote
\[
 \lambda^{ ( 1 ) } = \lambda_0,  \quad \lambda^{ ( 2 ) } = - \lambda_0,  \quad \lambda^{ ( 3 ) } = \lambda_1,  \quad \lambda^{ ( 4 ) } = - \lambda_1,  \quad \lambda^{ ( 5 ) } = \lambda_2,  \quad \lambda^{ ( 6 ) } = - \lambda_2 .
\]
(Note that they are not necessarily different.) Now we split the integral $ I_1 $ (see \eqref{I_1}) into the sum of the following six integrals
\bee
I_1 & =&  \frac{ i }{ 2t } \sum_{ k = 1 }^{ 6 } \int\limits_{ \partial B_{ \varepsilon }( \lambda^{ ( k ) } ) \backslash B_1( 0 ) } \frac{ f( \lambda ) e^{ t S( u, \lambda ) } }{ S_{ \lambda }( \lambda ) } d \bar \lambda \\
&  =:&  \sum_{ k = 1 }^{ 6 } I_1^{ k }.
\eee
Further, we split $ \Omega \backslash D_{ \varepsilon } $ into six domains, for $k = 1, \ldots, 6$,
\begin{equation*}
\Omega^{ (k) } = \big\{ \lambda \in \Omega \backslash ( D_{ \varepsilon } \cup B_1( 0 ) ) \ \colon  | \lambda - \lambda^{ ( k ) } | \leq | \lambda - \lambda^{ ( \ell ) } |, ~ \ell \neq k,~  \ell = 1, \ldots, 6 \big\}.
\end{equation*}
Following the previous case, we define
\begin{gather*}
I_2^{ -,k } = \frac{ 1 }{ t } \int\limits_{ \Omega^{ ( k ) } } \frac{ f_{ \la }( \lambda ) e^{ t S( u, \lambda ) } }{ S_{ \lambda }( u, \lambda ) }, \\
I_3^{-,k} = \frac{ 1 }{ t } \int\limits_{ \Omega^{ ( k ) } } \frac{ f( \lambda ) S_{ \lambda \lambda }( u, \lambda ) e^{ t S( u, \lambda ) } }{ ( S_{ \lambda }( u, \lambda ) )^2 }.
\end{gather*}
For every integral over domain $ \Omega^{ ( k ) } $, we perform the change of variables (see previous case for details)
\[
\lambda = \lambda^{ (k) } + \rho e^{ i \varphi } .
\]

\medskip

Now we compute each integral. Fix $\rho>0$. On $ \partial B_{ \rho }(\lambda^{ (k) }) \cap \Omega^{ (k) } $, we use the following estimates:
\[
| \lambda |^4 - 1 \lesssim \rho, \quad | \lambda |^2 - 1  \lesssim \rho, \quad | \lambda | \sim 1,
\]
and
\[
 |f( \lambda )| \lesssim \rho^{\al+1}, \quad |f_\la( \lambda )| \lesssim (1+|\bt|) \rho^{\al}.
\]
On the other hand,\footnote{In the worst case, where three stationary points coincide, this estimate is optimal.  Note also that the other three stationary points are antipodal, lying at distances $\sim 1$.}
\bee
| S_{ \lambda } | &\sim & \frac{ 1 }{ |\lambda|^4 }  |\lambda - \lambda^{(1)}|  |\lambda - \lambda^{(2)}| |\lambda - \lambda^{(3)}|  |\lambda - \lambda^{(4)}|   |\lambda - \lambda^{(5)}||\lambda - \lambda^{(6)}|  \\
& \gtrsim & \rho^3.
\eee
Finally,
\bee
 \left| \frac{ S_{ \lambda \lambda } }{ ( S_{ \lambda } )^2 } \right| & =& \abs{  \frac1{S_\la}  \Big[ \sum_{ j = 1}^6  \frac{ 1 }{  \lambda - \lambda^{(j)}  }   - 4 \frac{ 1}{ \lambda } \Big] }\\
 &  \lesssim &\frac{ 1 }{ \rho^4 }.
\eee
Using these observations, and the inclusion $\partial B_{ \varepsilon }( \lambda^{ ( k ) } ) \backslash B_1( 0 ) \subset \Omega^{(k)}$, we are able to obtain the following estimates for each $k$:
\bea
 | I_{ 1 }^{ k } | & \lesssim &  \frac1t \int\limits_{ \partial B_{ \varepsilon }( \lambda^{ ( k ) } ) \backslash B_1( 0 ) } \frac{ |f( \lambda )| } { |S_{ \lambda }( \lambda )| } d \bar \lambda  \nonu \\
 & \lesssim & \frac{ 1 }{ t } \frac{ \varepsilon^{ \alpha + 1 } \varepsilon }{ \varepsilon^3 } = \frac{ 1 }{ t } \varepsilon^{ \alpha - 1 }; \label{I_1k_case2}
\eea
for $ j = 1, 2, 3,$
\bea
 | I_2^{ -,k } | & \lesssim & \frac{ 1 }{ t } \int\limits_{ \Omega^{ ( k ) } } \frac{ |f_{ \la }( \lambda )| }{ |S_{ \lambda }( u, \lambda )| }  \nonu\\
 & \lesssim &  \frac1t (1+\bt)\int\limits_{ \varepsilon }^{ \varepsilon_0 } \frac{ \rho^{ \alpha } \rho d \rho }{ \rho^3 } \lesssim \frac{ 1 }{ t } (1+|\bt|)\varepsilon^{ \alpha - 1 } , \label{I_2jmk_case2}
\eea
and
\bea
 | I_3^{-,k} | & \lesssim & \frac{ 1 }{ t } \int\limits_{ \Omega^{ ( k ) } } \frac{ |f( \lambda )||S_{ \lambda \lambda }( u, \lambda )| } { | S_{ \lambda }( u, \lambda ) |^2 } \nonu\\
 & \lesssim & \frac1t \int\limits_{ \varepsilon }^{ \varepsilon_0 } \frac{ \rho^{ \alpha + 1 } \rho d \rho }{ \rho^4 } \lesssim \frac{ 1 }{ t } \varepsilon^{ \alpha - 1 }, \label{I_3mk_case2}
\eea
where $ \varepsilon_0 $ is a certain fixed constant, independent of $u$ and $t$.

\medskip

Gathering estimates \eqref{I_int_case2}, \eqref{I_2_jp_case2}, \eqref{I_3_p_case2} and \eqref{I_1k_case2}-\eqref{I_3mk_case2}, we obtain
\[
|I| \lesssim \ve^{\al+3} + \frac 1t  \Big( 1+  (1+|\bt|)\varepsilon^{ \alpha - 1 } \Big).
\]
Finally, for $t>0$ large we choose $ \varepsilon := \min \Big\{ \delta_0, \frac{ 1 }{ t^{ 1/4 }  } \Big\} $ to obtain that, if $ \alpha < 1 $, then
\begin{equation*}
| I | \lesssim (1+|\bt|) \Big[ \frac1t + \frac{ 1 }{ t^{  (\alpha + 3 )/4  } } \Big] \lesssim \frac{  (1+|\bt|) }{ t^{  (\alpha + 3 )/4  } }.
\end{equation*}

\subsubsection{Case 3}

For the sake of easiness, we recall the definition of case 3: the region $ u \in \mathbb{C} \backslash \mathbb{U} $ and $|\la_0| = 1 + \omega < 2 $ (see Lemma \ref{st_points_lemma}, item 4 and Fig. \ref{Case3}).

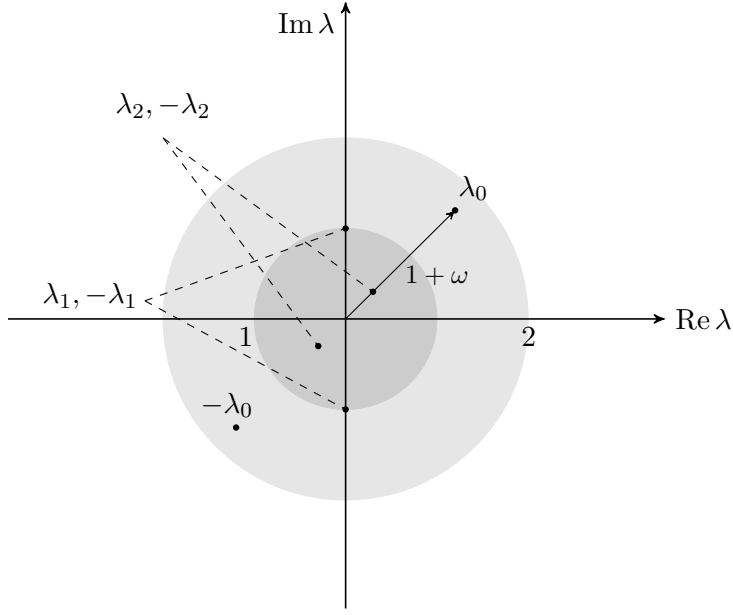
\begin{figure}
\begin{center}
\begin{tikzpicture}[
	>=stealth',
	axis/.style={semithick,->},
	coord/.style={dashed, semithick},
	yscale = 1.2,
	xscale = 1.2]
	\newcommand{\xmin}{-3.5};
	\newcommand{\xmax}{3.5};
	\newcommand{\ymin}{-3};
	\newcommand{\ymax}{3.5};
	\newcommand{\ta}{3};
	\newcommand{\fsp}{0.2};
	\filldraw[color=light-gray1] (0,0) circle (2);
	\filldraw[color=light-gray2] (0,0) circle (1);
	\draw [axis] (\xmin-\fsp,0) -- (\xmax,0) node [right] {$\re \la$};
	\draw [axis] (0,\ymin-\fsp) -- (0,\ymax) node [below left] {$\ima \la$};
	\draw (-0.9,-0.2) node [left] {$1$};
	\draw (2.2,-0.2) node [left] {$2$};
	\draw (1.4,1.2) node [above] {$\la_0$};
	\draw (-1.3,-1.2) node [above] {$-\la_0$};
	\draw (-2,2.1) node [above] {$\la_2, -\la_2$};
	\draw (-2.8,0) node [above] {$\la_1,-\la_1$};
	\draw [dashed] (-2,2) -- (0.3,0.3);
	\draw [dashed] (-2,2) -- (-0.3,-0.3);
	\draw [dashed] (-2.2,0.2) -- (0,1);
	\draw [dashed] (-2.2,0.2) -- (0,-1);
	\fill (0.3,0.3)  circle[radius=1pt];
	\fill (-0.3,-0.3)  circle[radius=1pt];
	\fill (0,1)  circle[radius=1pt];
	\fill (0,-1)  circle[radius=1pt];
	\fill (1.2,1.2)  circle[radius=1pt];
	\fill (-1.2,-1.2)  circle[radius=1pt];
	\draw [->] (0,0) -- (1.2,1.2);
	\draw (1,0.7) node [below] {\small $1+\omega$};
\end{tikzpicture}
\end{center}
\caption{The setting for Case 3. Note that if e.g. $\la_0,\la_1$ and $\la_2$ are close, and $\la$ is close to them, then $-\la_0,-\la_1$ and $-\la_2$ are at distance $\sim 1$ of $\la$. This fact is useful to get a lower estimate on $S_\la$.}\label{Case3}
\end{figure}
\noindent

\medskip

As usual, we take $ D_{ \varepsilon } $ to be the union of disks with radius $ \varepsilon $ and with centers in the stationary points. For $t>0$ large enough, we consider the following two sub-cases:
\[
0 < \omega \leq \frac{ 4 }{ t^{ 1/4 } },  \quad \hbox{and} \quad  \omega > \frac{ 4 }{ t^{ 1/4 } }.
\]

\begin{itemize}
\item[I/] $ 0 < \omega \leq \frac{ 4 }{ t^{ 1/4 } } $.

In this subcase we set $ \varepsilon = \frac{ 1 }{ t^{ 1/4 } } $ thus obtaining that $ \omega \lesssim \varepsilon $.

We note that the following estimates are valid on $ D_{ \varepsilon } \backslash B_1( 0 ) $:
\begin{equation*}
( | \lambda |^2 - 1 )^{ \alpha } \lesssim ( \omega + \varepsilon )^{ \alpha } \lesssim \varepsilon^{ \alpha }, \quad | \lambda |^4 - 1 \lesssim \omega + \varepsilon \lesssim \varepsilon.
\end{equation*}
Further, it is easy to see that the reasoning used in the previous case applies to this case also. Thus we obtain that $ | I | \lesssim \frac{ 1 }{ t^{ \frac{ \alpha + 3 }{ 4 } } } $.

\item[II/] $ \omega > \frac{ 4 }{ t^{ 1/4 } } $.

In this case we set $ \varepsilon := \frac{ 1 }{ t^{ \frac{ \alpha + 3 }{ 8 } } } $ thus obtaining that
\[
\omega > 4 \varepsilon^{ \frac{ 2 }{ \alpha + 3 } } > 4 \varepsilon,
\]
since $ \frac{ 2 }{ \alpha + 3 } < 1 $.

The integral $ I_{ int } $ from \eqref{asp_int_sum} is directly estimated as follows:
\begin{equation*}
| I_{ int } | \lesssim  \int\limits_{ D_{ \varepsilon } \backslash B_1( 0 ) }| f( \lambda )| \lesssim  \varepsilon^2 \lesssim \frac{ 1 }{ t^{ \frac{ \alpha + 3 }{ 4 } } }.
\end{equation*}

Similarly, the integrals $ I_{2}^{ + } $, $ I_{3}^{+} $ are estimated as in Case 2, $ u \in \mathbb{U} $.

Now we pass to estimating integrals $ I_1 $, $ I_2^- $, $ I_3^- $. Similarly to the previous case we split each of these integrals into a sum of six integrals and perform the change of variables of the form $ \lambda = \lambda^{ ( j ) } + \rho e^{ i \varphi } $.

In order to estimate $ I_1 $ we use the following estimates valid on $ \partial B_{ \varepsilon }(\lambda^{ (k) }) \cap \Omega^{ ( k ) } $:
\begin{equation*}
( | \lambda |^2 - 1 )^{ \alpha } \lesssim ( \omega + \varepsilon )^{ \alpha } \lesssim \omega^{ \alpha }, \quad | \lambda |^4 - 1 \lesssim \omega + \varepsilon \lesssim \omega,
\end{equation*}
\[
|f(\la)| \lesssim \omega^{\al+1}, \qquad |f_\la(\la)| \lesssim (1+|\bt|)\omega^{\al},
\]
and\footnote{The worst case corresponds to $\omega$ being close to 1, and $|\la_2|$ is chosen close to 1}
\[
 | S_{ \lambda } |\sim   \frac{ 1 }{ |\lambda|^4 }  |\lambda - \lambda^{(1)}|  |\lambda - \lambda^{(2)}| |\lambda - \lambda^{(3)}|  |\lambda - \lambda^{(4)}|   |\lambda - \lambda^{(5)}||\lambda - \lambda^{(6)}|   \gtrsim \varepsilon \omega^2.
\]
This allows us to obtain
\bee
| I_1 | & \lesssim&  \frac{ 1 }{ t } \frac{ \omega^{ \alpha + 1 } \varepsilon }{ \varepsilon \omega^2 } \\
& \lesssim &  \frac{ 1 }{ t \omega^{ 1 - \alpha } } \\
& \lesssim  & \frac{ 1 }{ t \varepsilon^{ \frac{2 ( 1 - \alpha )}{\al+3} } } \\
& = &  \frac 1{t^{1- \frac{1-\al}{4}}} \lesssim  \frac{ 1 }{ t^{ \frac{ \alpha + 3 }{ 4 } } }.
\eee

In order to estimate $ I_2^{ -, k } $, $ I_3^{-,k} $ we use the following estimates on $ \partial B_{ \rho }(\la^{(k)}) \cap \Omega^{ ( k ) } $:
\begin{align*}
& ( | \lambda |^2 - 1 )^{ \alpha } \lesssim \begin{cases} & \omega^{ \alpha }, \text{ if } 4 \rho < \omega \\ & \rho^{ \alpha }, \text{ if } 4 \rho \geq \omega \end{cases}, \quad | \lambda |^4 - 1 \lesssim \begin{cases} & \omega, \text{ if } 4 \rho < \omega \\ & \rho, \text{ if } 4 \rho \geq \omega \end{cases}.
\end{align*}
For the term $S_{ \lambda }$, we use the fact that  in the considered region,
\bee
S_{ \lambda } & \sim & \frac{ 1 }{ |\lambda|^4 }  |\lambda - \lambda^{(1)}|  |\lambda - \lambda^{(2)}| |\lambda - \lambda^{(3)}|  |\lambda - \lambda^{(4)}|   |\lambda - \lambda^{(5)}||\lambda - \lambda^{(6)}|,
\eee
to conclude that
\begin{align*}
& | S_{ \lambda } | \gtrsim \begin{cases} & \rho \omega^2, \text{ if } 4 \rho < \omega \\ & \rho^3, \text{ if } 4 \rho \geq \omega \end{cases}.
\end{align*}
Finally, following the same approach as before, we have
\begin{align*}
\left| \frac{ S_{ \lambda \lambda } }{ ( S_{ \lambda } )^2 } \right| \lesssim \begin{cases} & \frac{ 1 }{ \rho^2 \omega^2}, \text{ if } 4 \rho < \omega \\ & \frac{ 1 }{ \rho^4 }, \text{ if } 4 \rho \geq \omega \end{cases}.
\end{align*}
In this manner we get the following
\begin{equation*}
| I_2^{ -,k } | \lesssim \frac{ 1 }{ t } \int\limits_{ \omega / 4 }^{ \varepsilon_0 } \frac{ \rho^{ \alpha + 1 } d \rho }{ \rho^3 } + \frac{ 1 }{ t } \int\limits_{ \varepsilon }^{ \omega / 4 } \frac{ \omega^{ \alpha } \rho d \rho }{ \rho \omega^2 } \lesssim \frac{ 1 }{ t \omega^{ 1 - \alpha } } \lesssim \frac{ 1 }{ t^{ \frac{ \alpha + 3 }{ 4 } } }.
\end{equation*}
and
\begin{equation*}
| I_3^{-, k } | \lesssim \frac{ 1 }{ t } \int\limits_{ \omega / 4 }^{ \varepsilon_0 } \frac{ \rho^{ \alpha + 2 } d \rho }{ \rho^4 } + \frac{ 1 }{ t } \int\limits_{ \varepsilon }^{ \omega / 4 } \frac{ \omega^{ \alpha + 1 } \rho d \rho }{ \rho^2 \omega^2 } \lesssim \frac{ | \ln \varepsilon | }{ t \omega^{ 1 - \alpha } } \lesssim \frac{ | \ln t | }{ t^{ \frac{ \alpha + 3 }{ 4 } } }.
\end{equation*}

Thus we have proved that in this case

\begin{equation*}
| I | \lesssim \frac{ | \ln t | }{ t^{ \frac{ \alpha + 3 }{ 4 } } },
\end{equation*}
as desired.
\end{itemize}

\subsection{Proof of the estimate for small $ t $}
\label{small_t_subsection}
We will suppose that $ t < e^{ - \frac{ 3 }{ 1 - \alpha } } $ (which implies, in particular, that $ t < 1 $). Note that in this case $ \frac{ 1 }{ t^{ \frac{ 1 }{ 3 } } } > e^{ \frac{ 1 }{ 1 - \alpha } } $ and thus for all $ \omega \geq \frac{ 1 }{ t^{ \frac{ 1 }{ 3 } } } $ we have that $ \frac{ \ln \omega }{ \omega^{ 1 - \alpha } } $ is a decreasing function of $ \omega $ and thus the following inequality is true: $ \frac{ \ln \omega }{ \omega^{ 1 - \alpha } } \leq \frac{ | \ln t | }{ 3 t^{ \frac{ \alpha - 1 }{ 3 } } } $ for all $ \omega $ satisfying the above-mentioned condition.\footnote{This assumption is used to treat $ \ln \omega $ arising in (\ref{log_one})-(\ref{log_two}). It can be replaced by the assumption that $ t < 1 $. Then, we can note, for example, that for all $ \varepsilon > 0 $ and $ \omega \geq 1 $ we have that $ \ln \omega \lesssim \omega^{ \varepsilon } $. Then at the end we obtain the following estimate: $ | I | \lesssim \frac{ 1 }{ t^{ \frac{ \alpha + \varepsilon + 2 }{ 3 } } } $ . Then, if $ \varepsilon < \frac{ 1 - \alpha }{ 4 } $, we also have that $ | I | \lesssim \frac{ 1 }{ t^{ \frac{ \alpha + 3 }{ 4 } } } $.}

We consider the following cases for the values of parameter $ u $.
\begin{enumerate}
\item All the stationary points lie in $ B_2( 0 ) $.
\item Stationary points $ \lambda_0 $, $ - \lambda_0 $ lie outside the ball $ B_2( 0 ) $.
\end{enumerate}

\subsubsection{Case 1}
In this case we set $ D_r = B_r( 0 ) $ with $ r = \frac{ 4 }{ t^{ 1/3 } } $ (note that $ r \geq 4 $) and we split the integral $ I $ into two parts:
\begin{equation*}
I = \int\limits_{ D_r \backslash B_1( 0 ) } + \int\limits_{ \mathbb{C} \backslash D_r } : = I_{ int } + I_{ ext }.
\end{equation*}
Evidently, we have
\begin{equation*}
| I_{ int } | \lesssim r^{ \alpha + 2 }.
\end{equation*}
Integrating $ I_{ ext } $ by parts, we represent it as a sum of three integrals (see page 12):
\begin{equation*}
I_{ ext } = I_1 + I_2 + I_3.
\end{equation*}
On $ \mathbb{C} \backslash D_r $ the following is true:
\begin{equation*}
| S_{ \lambda } | \sim | \lambda |^2.
\end{equation*}
Indeed,
\begin{align*}
& \left| \frac{ \lambda - \lambda_j }{ \lambda } \right| \geq 1 - \frac{ | \lambda_j | }{ | \lambda | } \geq 1 - \frac{ 2 }{ r } \geq \frac{ 1 }{ 2 }, \\
& | \lambda - \lambda_j | \geq | \lambda | - | \lambda_j | \geq | \lambda | - 2 \geq | \lambda | - \frac{ | \lambda | }{ 2 } = \frac{ | \lambda | }{ 2 }.
\end{align*}
In a similar way we have that
\begin{equation*}
\left| \frac{ S_{ \lambda \lambda } }{ ( S_{ \lambda } )^2 } \right| \lesssim \frac{ 1 }{ | \lambda |^3 }
\end{equation*}
This allows us to obtain the following estimates:
\begin{align*}
& | I_1 | \lesssim \frac{ r^{ \alpha } }{ t r }, \\
& | I_2 | \lesssim \frac{ 1 }{ t } \int\limits_{ r }^{ \infty } \frac{ \rho^{ \alpha - 1 } \rho }{ \rho^2 } d \rho = \frac{ 1 }{ t r^{ 1 - \alpha } }, \\
& | I_3 | \lesssim \frac{ 1 }{ t } \int\limits_{ r }^{ \infty } \frac{ \rho^{ \alpha } \rho }{ \rho^3 } d \rho = \frac{ 1 }{ t r^{ 1 - \alpha } }.
\end{align*}
Finally, taking into account that $ r \sim \frac{ 1 }{ t^{ 1/3 } } $, we get that $ | I | \lesssim \frac{ 1 }{ t^{ \frac{ \alpha + 2 }{ 3 } } } $.

\subsubsection{Case 2}
In this case we consider two subcases
\begin{enumerate}
\item $ 1 \leq \omega < \frac{ 1 }{ t^{ 1/3 } } $
In this case we take $ r = \frac{ 4 }{ t^{ 1/3 } } $ and set $ D_r = B_r( 0 ) $. Note that on $ \mathbb{C} \backslash D_r $ the following estimates hold:
\begin{align*}
& \frac{ | \lambda - \lambda_j | }{ | \lambda | } \geq 1 - \frac{ | \lambda_j | }{ | \lambda | } \geq 1 - \frac{ 1 + \omega }{ r } \geq \frac{ 1 }{ 2 }, \\
& | \lambda - \lambda_j | \geq | \lambda | - | \lambda_j | \geq | \lambda | - ( 1 + \omega ) \geq | \lambda | - \frac{ r }{ 2 } \geq \frac{ | \lambda | }{ 2 }.
\end{align*}
Thus, in this case the estimate can be carried out by using the reasoning of the previous case.

\item $ \frac{ 1 }{ t^{ 1/3 } } \leq \omega $

In this case we choose $ r = \min\left( \frac{ 1 }{ t^{ \frac{ \alpha + 2 }{ 6 } } \omega^{ \frac{ \alpha }{ 2 } } }, \frac{ 1 }{ 4 } \right) $ and we set $ D_r $ to be the union of $ B_{ 1 + \frac{ 1 }{ 4 } }( 0 ) $ and $ B_{ r }( \pm \lambda_0 ) $. Note that if $ \frac{ 1 }{ t^{ \frac{ \alpha + 2 }{ 6 } } \omega^{ \frac{ \alpha }{ 2 } } } > \frac{ 1 }{ 4 } $ (and hence $ r =\frac{ 1 }{ 4 } $), we have that $ \omega^{ \alpha } \lesssim \frac{ 1 }{ t^{ \frac{ \alpha + 2 }{ 3 } } } $. Further, we split $ I $ into $ I_{ int } $ and $ I_{ ext } $ as above. First of all, note that
\begin{equation*}
| I_{ int } | \lesssim 1 + \omega^{ \alpha } r^2 \lesssim \frac{ 1 }{ t^{ \frac{ \alpha + 2 }{ 3 } } }.
\end{equation*}

Now we introduce the following notation: $ \mathcal{B}_{ \rho } = B_{ 1 + \rho }( 0 ) \cup B_{ \rho }( \pm \lambda_0 ) $. Note that
\begin{equation*}
\mathbb{C} \backslash D_{ r } = ( \mathbb{C} \backslash \mathcal{B}_{ 1/4 } ) \cup ( B_{ 1/4 }( \pm \lambda_0 ) \backslash B_r( \pm \lambda_0 ) ).
\end{equation*}
Each of the integrals $ I_k $, $ k = 1, 2, 3 $, is split into the sum of two integrals $ I_k^- + I_k^+ $, where $ I_k^- $ is the integral over the set $ B_{ 1/4 }( \pm \lambda_0 ) \backslash B_r( \pm \lambda_0 ) $ and $ I_k^+ $ is the integral over $ \mathbb{C} \backslash \mathcal{B}_{ 1/4 } $. \footnote{To be more precise, if $ r < \frac{ 1 }{ 4 } $, then $ I_1^- $ is the integral over $  \partial D_r \cap B_{ 1/4 }( \pm \lambda_0 ) \backslash B_r( \pm \lambda_0 ) $, i.e. over $ \partial B_r( \pm \lambda_0 ) $ and $ I_1^+ $ is the integral over $ \partial D_r \cap \mathbb{C} \backslash \mathcal{B}_{ 1/4 } $, i.e. over $ \partial B_{ 1 + 1/4 }( 0 ) $. Note that if $ r \geq \frac{ 1 }{ 4 } $, then $ I_k \equiv I_k^+ $ and $ I_k^- $ are formally not defined in this case, $ k = 1, 2, 3 $.}

Note that on $ \partial \mathcal{B}_{ \rho } $ with $ \rho \geq \frac{ 1 }{ 4 } $ and on $ \partial B_{ \rho }( \pm \lambda_0 ) $ with $ r \leq \rho < \frac{ 1 }{ 4 } $, the following estimates are true:
\begin{equation*}
\frac{ | \lambda - \lambda_{st} | }{ | \lambda | } \geq 1 - \frac{ | \lambda_{st} | }{ | \lambda | } \gtrsim 1,
\end{equation*}
where $ \lambda_{st} = \pm \lambda_{ 1 } $ or $ \lambda_{st} = \pm \lambda_{ 2 } $. Thus we have that on $ \partial \mathcal{B}_{ \rho } $ with $ \rho \geq \frac{ 1 }{ 4 } $ and on $ \partial B_{ \rho }( \pm \lambda_0 ) $ with $ r \leq \rho < \frac{ 1 }{ 4 } $ the following holds: $ | S_{ \lambda } | \gtrsim | \lambda - \lambda_0 || \lambda + \lambda_0 | $. Using this fact we can obtain the following estimates valid on $ \partial \mathcal{B}_{ \rho } $ with $ \rho \geq \frac{ 1 }{ 4 } $ and on $ \partial B_{ \rho }( \pm \lambda_0 ) $ with $ r \leq \rho < \frac{ 1 }{ 4 } $:
\begin{equation*}
| S_{ \lambda } | \gtrsim \begin{cases} & \omega \rho, \quad 4 \rho < \omega, \\ & \rho^2, \quad 4 \rho \geq \omega, \end{cases} \quad \left| \frac{ S_{ \lambda \lambda } }{ ( S_{ \lambda } )^2 } \right| \lesssim \begin{cases} & \frac{ 1 }{ \omega \rho^2 }, \quad 4 \rho < \omega, \\ & \frac{ 1 }{ \rho^3 }, \quad 4 \rho \geq \omega. \end{cases}
\end{equation*}

\begin{figure}
\begin{center}
\begin{tikzpicture}[
	>=stealth',
	axis/.style={semithick,->},
	coord/.style={dashed, semithick},
	yscale = 2.0,
	xscale = 2.0]
	\newcommand{\xmin}{-2.0};
	\newcommand{\xmax}{2.0};
	\newcommand{\ymin}{-2.0};
	\newcommand{\ymax}{2.0};
	\newcommand{\ta}{3};
	\newcommand{\fsp}{0.2};
	\filldraw[color=light-gray2] (0,0) circle (1.25);
	\filldraw[color=light-gray1] (0,0) circle (1);
    \draw [dotted] (0,0) circle (1.45);
    \filldraw[color=light-gray2] (1.4,1.4) circle (0.1);
    \draw (1.4,1.4) circle (0.17);
    \draw [dashed] (1.4,1.4) circle (0.25);
    \draw [dotted] (1.4,1.4) circle (0.45);
    \filldraw[color=light-gray2] (-1.4,-1.4) circle (0.1);
    \draw (-1.4,-1.4) circle (0.17);
    \draw [dashed] (-1.4,-1.4) circle (0.25);
    \draw [dotted] (-1.4,-1.4) circle (0.45);
	\draw [axis] (\xmin-\fsp,0) -- (\xmax,0) node [right] {$\re \la$};
	\draw [axis] (0,\ymin-\fsp) -- (0,\ymax) node [below left] {$\ima \la$};
	\draw (-0.9,-0.2) node [left] {$1$};
	\draw (1.9,0.9) node [right] {$\la_0$};
	\draw (-1.9,-0.9) node [left] {$-\la_0$};
	\draw (-1.5,1.6) node [above] {$\la_2, -\la_2$};
	\draw (-2.2,0.1) node [above] {$\la_1,-\la_1$};
	\draw [dashed] (-1.5,1.5) -- (0.3,0.3);
	\draw [dashed] (-1.5,1.5) -- (-0.3,-0.3);
	\draw [dashed] (-1.8,0.2) -- (0,1);
	\draw [dashed] (-1.8,0.2) -- (0,-1);
    \draw [dashed] (1.4,1.4) -- (1.9,0.9);
    \draw [dashed] (-1.4,-1.4) -- (-1.9,-0.9);
	\fill (0.3,0.3)  circle[radius=1pt];
	\fill (-0.3,-0.3)  circle[radius=1pt];
	\fill (0,1)  circle[radius=1pt];
	\fill (0,-1)  circle[radius=1pt];
	\fill (1.4,1.4)  circle[radius=1pt];
	\fill (-1.4,-1.4)  circle[radius=1pt];
	\draw [->] (0,0) -- (1.4,1.4);
\end{tikzpicture}
\end{center}
\caption{The setting for Case 2, subcase 2. The set $ \mathcal{B}_r $ is the shadowed region. The black circles form $ \partial B_{ \rho }( \pm \lambda_0 ) $ with $ r \leq \rho < \frac{ 1 }{ 4 } $, the dashed circles are $ \partial B_{ 1/4 }( \pm \lambda_0 ) $ and the dotted circles form $ \partial \mathcal{B}_{ \rho } $ with $ \rho \geq \frac{ 1 }{ 4 } $.}
\end{figure}
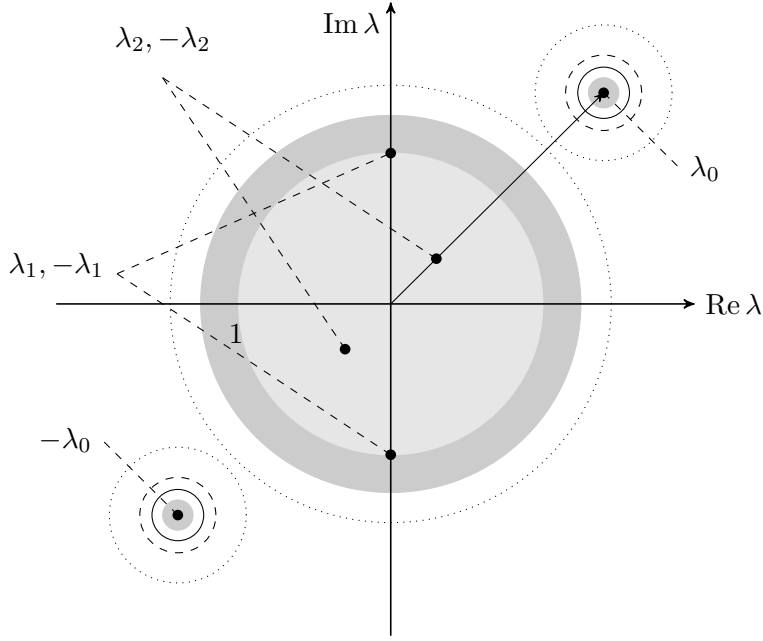
\noindent

Using these estimates, we are brought to the following estimates on the parts of $ I_{ ext } $:
\begin{align*}
& | I_1^+ | \lesssim \frac{ \omega^{ \alpha } }{ t \omega } = \frac{ 1 }{ t \omega^{ 1 - \alpha } } \leq \frac{ 1 }{ t^{ \frac{ \alpha + 2 }{ 3 } } }, \\
& | I_1^- | \lesssim \frac{ \omega^{ \alpha } r }{ t \omega r } = \frac{ 1 }{ t \omega^{ 1 - \alpha } } \leq \frac{ 1 }{ t^{ \frac{ \alpha + 2 }{ 3 } } }, \end{align*}
\begin{align*}
& | I_2^+ | \lesssim \int\limits_{ 1/4 }^{ \omega / 4 } \frac{ \omega^{ \alpha - 1 } \rho d \rho }{ t \omega \rho } + \int\limits_{ \omega / 4 }^{ \infty } \frac{ \rho^{ \alpha - 1 } \rho d \rho }{ t \rho^2 } \lesssim \frac{ 1 }{ t \omega^{ 1 - \alpha } } \leq \frac{ 1 }{ t^{ \frac{ \alpha + 2 }{ 3 } } }, \\
& | I_2^- | \lesssim \int\limits_{ r }^{ 1/4 } \frac{ \omega^{ \alpha - 1 } \rho d \rho }{ t \omega \rho } \lesssim \frac{ 1 }{ t \omega^{ 2 - \alpha } } + \frac{ r }{ t \omega^{ 2 - \alpha } } \leq \frac{ 1 }{ t^{ \frac{ \alpha + 2 }{ 3 } } },
\end{align*}
\begin{align}
\label{log_one}
& | I_3^+ | \lesssim \int\limits_{ 1/4 }^{ \omega / 4 } \frac{ \omega^{ \alpha } \rho d \rho }{ t \omega \rho^2 } + \int\limits_{ \omega / 4 }^{ \infty } \frac{ \rho^{ \alpha } \rho d \rho }{ t \rho^3 } \lesssim \frac{ \ln \omega }{ t \omega^{ 1 - \alpha } } \lesssim \frac{ | \ln t | }{ t^{ \frac{ \alpha + 2 }{ 3 } } }, \\
\label{log_two}
& | I_3^- | \lesssim \int\limits_{ r }^{ 1 / 4 } \frac{ \omega^{ \alpha } \rho d \rho }{ t \omega \rho^2 } \lesssim \frac{ | \ln r | }{ t \omega^{ 1 - \alpha } } \lesssim \frac{ | \ln t | }{ t^{ \frac{ \alpha + 2 }{ 3 } } }.
\end{align}

\end{enumerate}

\section{Bilinear Estimates}

The purpose of this section is to show bilinear estimates for the NV equation for negative energy, always taking into account the size of the fixed energy.

In this section we use the following notations. We denote by $ \langle f\rangle $ the japanese bracket:
\[
\langle f\rangle := (1+|f|^2)^{1/2};
\]
$\mathcal F[ u ]$ and $ \hat u $ both denote the Fourier transform of $ u $ in $(t,x,y)$; we also have $ A \land B : = \min( A, B ) $ and $ A \lor B : = \max( A, B ) $. Variables $ N, \tilde N, \check N, L, \tilde L, \check L $ of this section are \emph{dyadic}, i.e. their range is $ \{ 2^k, k \in \mathbb{N} \} $.

\medskip

Now we introduce the associated $X^{s,b}$ spaces \cite{B1} for the NV dynamics.  Let $\tilde\varphi \in C_0^\infty(\R;[0,1])$ be a cutoff function such that
\[
\tilde\varphi(s) =0 \quad \hbox{ for } \quad |s| \geq 1, \qquad \tilde\varphi(s) = 1 \quad \hbox{ for } \quad |s|\leq \frac12.
\]
We define
\[
\varphi(s):= \tilde\varphi(s) -\tilde\varphi(2s).
\]
We introduce the frequency projection operators at a dyadic frequency $N >1$, as follows:
\[
\varphi_N (s):= \varphi(N^{-1} s),
\]
and for $N=1$,
\[
\varphi_1 (s):= \tilde\varphi(s).
\]
Using these multipliers, we have for the Fourier transform in the $\xi$ variable,
\[
P_N u := \mathcal F^{-1}\Big[\varphi_N (|E|^{-1/2} |\xi|) \mathcal F[u](\xi)\Big].
\]
Recall the definition of the phase $\tilde S(u,\xi)$ in \eqref{st_phase}. We define
\[
w(\xi,\bar\xi;E) := ( \xi^3 + \bar \xi^3 )\left( 1 - \frac{ 3 E }{ |\xi|^2 } \right).
\]
Note that the above expression is real-valued. In order to perform some Fourier analysis, we need $w$ in terms of real-valued coordinates. Put $\xi=\xi_1+i\xi_2$. We have
\be\label{w0}
w (\xi, \bar \xi) := w(\xi,\bar\xi; E) = 2( \xi_1^3 -3\xi_1\xi_2^2)\left( 1 + \frac{ 3|E|}{ \xi_1^2+\xi_2^2 } \right).
\ee
Compared with the work of Angelopoulus \cite{A}, the $NV_\pm$ symbol has a nontrivial component which becomes important at low frequencies\footnote{Although compared with the usual KP symbols, the NV symbols are still bounded at the origin.}.  We also define
\[
\sigma(\tau,\xi_1,\xi_2):= \tau -w(\xi_1,\xi_2).
\]
We introduce then
\[
Q_L u := \mathcal F^{-1}[ \varphi_L(|E|^{-3/2}|\sigma|) \mathcal F[u](\tau,\xi)].
\]
Finally, for a fixed energy $E$, we say that $u=u(t,x,y) \in X_E^{s,b}$ for $s,b\in \R$ if $u\in L^2(\R^3)$ and its Fourier transform $\hat u$ satisfies the integral condition
\be\label{Xsb}
\int_{\R^3} \langle \sigma\rangle^{2b} \langle |\xi| \rangle^{2s} |\hat u(\tau,\xi_1,\xi_2)|^2 d\tau d\xi_1d\xi_2 <+\infty.
\ee

\begin{prop}
Assume that $E<0$. Then we have for $\ve>0$ small and $s>\frac12$,
\be\label{Bilinear}
\|\partial_z(vw)\|_{X_E^{s,-1/2-2\ve}} \lesssim  |E|^{(3-4s)/8}  ~\|v\|_{X_E^{s,1/2+ \ve}} \|w\|_{X_E^{s,1/2+ \ve}},
\ee
for all $v,w$ such that the right hand side makes sense.
\end{prop}

\begin{proof}
We follow closely the ideas from \cite{MP}, using  a modified version of the original ideas by Bourgain \cite{B1} and Kenig, Ponce and Vega \cite{KPV1}. As usual, by duality we are lead to prove that
\[
J:= \int_{\R^6} K[\tau,\tilde\tau,\xi,\tilde\xi]\hat u(\tau,\xi_1,\xi_2)\hat v(\check\tau,\check\xi_1,\check\xi_2)\hat w(\tilde\tau,\tilde\xi_1,\tilde\xi_2)\, d\xi_1 \, d\tilde\xi_1\, d\xi_2 \, d\tilde\xi_2 \, d\tau \, d\tilde\tau
\]
satisfies
\[
|J| \lesssim \|u\|_{L^2} \|v\|_{L^2} \|w\|_{L^2},
\]
for any $u,v,w\in L^2(\R^3)$, and where for $j=1,2$, $\tilde\sigma := \sigma( \tilde\tau,\tilde \xi)$,
\be\label{hat_tilde}
\check\tau := \tilde\tau -\tau, \quad \check \xi: =\tilde \xi-\xi,\quad \check\sigma:= \tilde\sigma -\sigma.
\ee
Here the kernel $K=K[\tau,\tilde\tau,\xi,\tilde\xi]$ is explicitly given by
\[
K:=|i\tilde\xi_1+\tilde\xi_2| |i\check\xi_1+\check \xi_2||-i \check \xi_1+\check \xi_2|^{-1} \langle \tilde\sigma\rangle^{-1/2+2\ve}\langle \sigma\rangle^{-1/2-\ve}\langle \check \sigma\rangle^{-1/2-\ve}\langle |\tilde \xi| \rangle^{s} \langle |\xi|\rangle^{-s}\langle |\check \xi|\rangle^{-s}.
\]
A further simplification leads to
\be\label{KernelK}
K=|\tilde \xi|  \langle \tilde\sigma\rangle^{-1/2+2\ve}\langle \sigma\rangle^{-1/2-\ve}\langle \check\sigma\rangle^{-1/2-\ve}\langle |\tilde \xi| \rangle^{s} \langle |\xi|\rangle^{-s}\langle |\check \xi|\rangle^{-s}.
\ee
Now we use dyadic decompositions to split  $J$ into several pieces. We put
\[
J= \sum_{N,\tilde N,\check N} J_{N,\tilde N,\check N},
\]
where
\be\label{JNNN}
J_{N,\tilde N,\check N} := \int_{\R^6} K[\tau,\tilde\tau,\xi,\tilde\xi] \widehat{P_Nu}(\tau,\xi_1,\xi_2)\widehat{P_{\check N}v}(\check\tau,\check\xi_1,\check\xi_2)\widehat{P_{\tilde N}w}(\tilde\tau,\tilde\xi_1,\tilde\xi_2)d\xi_1d\tilde\xi_1d\xi_2d\tilde\xi_2 d\tau d\tilde\tau.
\ee

\noindent
\emph{Estimate for low-low to low frequencies.} This is the simplest case. Here we have $N\sim \tilde N \sim \check N$, where $N\sim 1$.  From \eqref{JNNN} we have to estimate the quantity
\[
\sum_{N\sim \tilde N \sim \check N \sim 1} J_{N,\tilde N,\check N}.
\]
Note that we also have $|\xi|\sim |\tilde \xi| \sim |\check \xi| \sim |E|^{1/2}$. Therefore, using \eqref{KernelK} we have
\bee
|K|  & \lesssim &   |E|^{(1-s)/2} \langle \tilde\sigma\rangle^{-1/2+2\ve}\langle\sigma\rangle^{-1/2-\ve}\langle \check\sigma\rangle^{-1/2-\ve}\\
& \lesssim &  |E|^{(1-s)/2} \langle\sigma\rangle^{-1/2-\ve}\langle \check\sigma\rangle^{-1/2-\ve}.
\eee
Therefore, from \eqref{JNNN}, Plancherel, and using Cauchy-Schwarz,
\[
|J_{N,\tilde N,\check N}| \lesssim  |E|^{(1-s)/2}  \norm{\mathcal F^{-1}[\langle\sigma\rangle^{-1/2-\ve}  \widehat{P_Nu}]   }_{L^4}\norm{\mathcal F^{-1}[\langle\check \sigma\rangle^{-1/2-\ve}   \widehat{P_{\check N}v} ] }_{L^4} \|P_{\tilde N}w\|_{L^2}.
\]
From \eqref{Smoothing1} and the definition of $X^{s,b}$ space norm, we obtain for $\ve>0$ small,
\bee
\norm{\mathcal F^{-1}[\langle\sigma\rangle^{-1/2-\ve}  \widehat{P_Nu}]   }_{L^4} & = & \norm{U(t)U(-t)\mathcal F^{-1}[\langle\sigma\rangle^{-1/2-\ve}  \widehat{P_Nu}]   }_{L^4}  \\
& \lesssim &  |E|^{-1/16^+}  \norm{U(t)U(-t)\mathcal F^{-1}[\langle\sigma\rangle^{-1/2-\ve}  \widehat{P_Nu}]   }_{X^{0,\frac7{16}+}} \\
& \lesssim &   |E|^{-1/16^+} \| \langle\sigma\rangle^{\frac7{16}^+} \langle\sigma\rangle^{-1/2-\ve}\widehat{P_Nu} \|_{L^2} \\
&  \lesssim &  |E|^{-1/16^+} \|P_Nu \|_{L^2}.
\eee
We conclude that
\[
\sum_{N\sim \tilde N \sim \check N \sim 1} J_{N,\tilde N,\check N} \lesssim  |E|^{(3^--4s)/8} \|u \|_{L^2}\|v \|_{L^2}\|w \|_{L^2}.
\]

\medskip

\noindent
\emph{Estimate for low-high to high frequencies.} The worst case here corresponds to $N\ll \tilde N$ and $\tilde N \sim \check N$, where $\tilde N \gg 1$.  From \eqref{JNNN} we have to estimate the quantity
\[
\sum_{N\ll \tilde N, \, \tilde N \sim \check N   \gg 1} J_{N,\tilde N,\check N}.
\]
Unfortunately, a crude estimate for $K$ in \eqref{KernelK} shows that it is not possible to counter balance the term $|\tilde \xi| \sim |E|^{1/2} \tilde N$.  For this reason a dyadic decomposition on the modulation variables is necessary. Using the frequency localization operator $Q_{L}$, $Q_{\tilde L}$ and $Q_{\check L}$ on dyadic shells $\sigma \sim |E|^{3/2} L$, $\tilde \sigma \sim |E|^{3/2}\tilde L$ and so on,\footnote{The power $3/2$ of the energy in front of the modulation variable $\sigma$ is dictated by the time scaling.} we have
\[
J_{N,\tilde N,\check N} =\sum_{L,\tilde L,\check L} J_{N,\tilde N,\check N}^{L,\tilde L,\check L},
\]
where
\bea\label{JNNNLLL}
J_{N,\tilde N,\check N}^{L,\tilde L,\check L} &  := & \int_{\R^6} K[\tau,\tilde\tau,\xi,\tilde\xi] \mathcal F[P_NQ_L u](\tau,\xi_1,\xi_2) \times \nonu\\
& & \qquad \times \mathcal F[P_{\check N}Q_{\check L}v](\check\tau,\check\xi_1,\check\xi_2)\mathcal F[P_{\tilde N}Q_{\tilde L}w](\tilde\tau,\tilde\xi_1,\tilde\xi_2)\, d\xi_1d\tilde\xi_1d\xi_2d\tilde\xi_2 d\tau d\tilde\tau.
\eea
We readily have in the considered region
\[
|K[\tau,\tilde\tau,\xi,\tilde\xi]| \lesssim |E|^{-(s+7)/4} \tilde N N^{-s}  \tilde L^{-1/2+2\ve}L^{-1/2-\ve}\check L^{-1/2-\ve},
\]
and using Cauchy-Schwarz, the fact that
\[
\int_{\tau,\xi,\tilde \tau,\tilde \xi}\mathcal F[P_NQ_L u](\tau,\xi_1,\xi_2)\mathcal F[P_{\check N}Q_{\check L}v](\check\tau,\check\xi_1,\check\xi_2)
\]
represents a convolution in Fourier variables, and Plancherel, we get
\be\label{Aux_J}
|J_{N,\tilde N,\check N}^{L,\tilde L,\check L}| \lesssim  |E|^{-(s+7)/4} \tilde N  N^{-s}\tilde L^{-1/2+2\ve}L^{-1/2-\ve}\check L^{-1/2-\ve} \| P_NQ_L u ~ P_{\check N}Q_{\check L}v\|_{L^2}\|P_{\tilde N}Q_{\tilde L}w\|_{L^2}.
\ee
Now we use the following estimate (see \cite{MP} for a similar statement), valid under the assumptions that we work with low-high to high frequencies:
\be\label{Strichartz_Mod}
\| P_NQ_L u ~ P_{\check N}Q_{\check L}w \|_{L^2} \lesssim |E|^{5/4} N^{1/2} \check N^{-1} L^{1/2} \check L^{1/2}\| P_NQ_L u\|_{L^2} \| P_{\check N}Q_{\check L}w\|_{L^2}.
\ee
This estimate is proved several lines below. For now we assume the validity of this estimate and we continue with the estimation of \eqref{Aux_J}. Note that \eqref{Strichartz_Mod} allows to cancel out the bad frequency $\tilde N$ in \eqref{Aux_J}. We get
\[
|J_{N,\tilde N,\check N}^{L,\tilde L,\check L}| \lesssim  |E|^{-(2+s)/4} N^{1/2-s}\tilde L^{-1/2+2\ve}L^{-\ve}\check L^{-\ve}\| P_NQ_L u\|_{L^2} \| P_{\check N}Q_{\check L}v\|_{L^2}\|P_{\tilde N}Q_{\tilde L}w\|_{L^2}.
\]
Adding up on $N$ ($s>\frac12$), $\tilde L$, $L$ and $\check L$ we obtain
\bee
\sum_{N\ll \tilde N, \,\tilde N \sim \check N \gg 1} J_{N,\tilde N,\check N} & \lesssim  & |E|^{-(2+s)/4} \|u\|_{L^2} \sum_{\tilde N\sim \check N}  \| P_{\check N}v\|_{L^2}\|P_{\tilde N}w\|_{L^2} \\
& \lesssim  & |E|^{-(2+s)/4} \|u\|_{L^2}\|v\|_{L^2}\|w\|_{L^2}.
\eee
Let us prove \eqref{Strichartz_Mod}. Following \cite{MP}, and using the Plancherel's identity, together with Young's inequality for convolutions, we have
\bea
\| P_NQ_L u ~ P_{\check N}Q_{\check L}w \|_{L^2} & \sim & \| \mathcal F[P_NQ_L u] \star \mathcal F[ P_{\check N}Q_{\check L}w] \|_{L^2} \nonu\\
& \lesssim & \sup_{(\tilde \tau,\tilde\xi)\in\R^3} (\meas A_E(\tilde\tau,\tilde\xi))^{1/2} \| P_NQ_L u\|_{L^2} \|P_{\check N}Q_{\check L}w\|_{L^2}, \label{Est_1}
\eea
where $A_E(\tilde\tau,\tilde\xi)$ is the set (see \eqref{hat_tilde})
\[
A_E(\tilde\tau,\tilde\xi):=\big\{(\tau,\xi) \in \R^3 \ : \ |\sigma| \sim |E|^{3/2} L, \; |\check \sigma| \sim  |E|^{3/2} \check L, \;  |\xi| \sim  |E|^{1/2} N,  \; |\check \xi| \sim  |E|^{1/2} \check N \big\}.
\]
The measure of this set can be estimated as follows:
\be\label{meas_A}
\meas A_E(\tilde\tau,\tilde\xi) \lesssim  |E|^{3/2}(L \land \check L)  \meas B_E(\tilde\tau,\tilde\xi),
\ee
where
\[
B_E(\tilde\tau,\tilde\xi) := \big\{ (\tau,\xi) \in \R^3 \ : \ |\tilde \sigma + H[\xi,\tilde \xi] | \lesssim  |E|^{3/2}(L\lor \check L), \;  |\xi| \sim  |E|^{1/2} N,  \; |\check \xi| \sim  |E|^{1/2} \check N \big\},
\]
and $H$ is the standard  resonance function (see \eqref{w0})
\[
H[\xi,\tilde \xi] := w(\tilde \xi, \overline{\tilde \xi}) - w(\xi, \overline \xi) -w(\tilde \xi -\xi,\overline{\tilde \xi} -\overline\xi).
\]
We use now the idea from \cite[Lemma 3.8]{MP}: we estimate the measure of $B_E$ by finding lower bounds on the derivatives of the function $\xi \mapsto  H[\xi, \tilde \xi]$ in the considered region.

\medskip

As in \cite{A}, we have to distinguish between two cases: $|\tilde \xi_1-\xi_1|\sim |\tilde \xi_2-\xi_2| \sim |E|^{1/2} \tilde N \sim |E|^{1/2} \check N$ and $|\tilde \xi_1 -\xi_1|\gg |\tilde \xi_2 -\xi_2| $ (the remaining case is identical). For the first case we have
\bee
\partial_{\xi_2} H[\xi,\tilde \xi] & =&  -2\partial_{\xi_2} \Bigg[  ( \xi_1^3 -3\xi_1\xi_2^2)\left( 1 + \frac{ 3|E|}{ \xi_1^2+\xi_2^2 } \right) \\
&& \qquad + ( (\tilde\xi_1-\xi_1)^3 -3(\tilde\xi_1-\xi_1)(\tilde\xi_2-\xi_2)^2)\left( 1 + \frac{ 3|E|}{ (\tilde\xi_1-\xi_1)^2+(\tilde\xi_2-\xi_2)^2 } \right)\Bigg] \\
& =& -12\Bigg[  -\xi_1\xi_2\left( 1 + \frac{ 3|E|}{ \xi_1^2+\xi_2^2 } \right)  +(\tilde\xi_1-\xi_1)(\tilde\xi_2-\xi_2)\left( 1 + \frac{ 3|E|}{ (\tilde\xi_1-\xi_1)^2+(\tilde\xi_2-\xi_2)^2 } \right) \\
& &\qquad    -\frac{ |E|\xi_1 \xi_2( \xi_1^2 -3\xi_2^2)}{( \xi_1^2+\xi_2^2)^2 }   + \frac{|E|(\tilde\xi_1-\xi_1)(\tilde\xi_2-\xi_2)  ( (\tilde\xi_1-\xi_1)^2 -3(\tilde\xi_2-\xi_2)^2)}{( (\tilde\xi_1-\xi_1)^2+(\tilde\xi_2-\xi_2)^2)^2 } \Bigg].
\eee
Using the estimate
\[
\abs{\frac{a b }{ a^2+ b^2 }} \leq \frac12,
\]
and similar other estimates for the fractional terms appearing from the fact that we work with nonzero energies, and valid for all $(a,b)\in \R^2$ (the limit at the origin is not well-defined, but the functions are always bounded), we get
\[
|\partial_{\xi_2} H[\xi,\tilde \xi]| \sim |E| \check N^2.
\]
In the second case, we have no problems since
\bee
\partial_{\xi_1} H[\xi,\tilde \xi] & =& -2\partial_{\xi_1} \Bigg[  ( \xi_1^3 -3\xi_1\xi_2^2)\left( 1 + \frac{ 3|E|}{ \xi_1^2+\xi_2^2 } \right) \\
&& \qquad + ( (\tilde\xi_1-\xi_1)^3 -3(\tilde\xi_1-\xi_1)(\tilde\xi_2-\xi_2)^2)\left( 1 + \frac{ 3|E|}{ (\tilde\xi_1-\xi_1)^2+(\tilde\xi_2-\xi_2)^2 } \right)\Bigg] \\
& =& -6 \Bigg[  ( \xi_1^2 - \xi_2^2)\left( 1 + \frac{ 3|E|}{ \xi_1^2+\xi_2^2 } \right) \\
&&\qquad  - ( (\tilde\xi_1-\xi_1)^2 - (\tilde\xi_2-\xi_2)^2) \left( 1 + \frac{ 3|E|}{ (\tilde\xi_1-\xi_1)^2+(\tilde\xi_2-\xi_2)^2 } \right) \\
& &\qquad    -\frac{ 2|E|\xi_1^2( \xi_1^2 -3\xi_2^2)}{( \xi_1^2+\xi_2^2)^2 }   + \frac{2|E|(\tilde\xi_1-\xi_1)^2  ( (\tilde\xi_1-\xi_1)^2 -3(\tilde\xi_2-\xi_2)^2)}{( (\tilde\xi_1-\xi_1)^2+(\tilde\xi_2-\xi_2)^2)^2 } \Bigg],
\eee
so that
\[
|\partial_{\xi_1} H[\xi,\tilde \xi]| \sim |E| \check N^2.
\]
We conclude (see \cite{MP} for example) that
\[
\meas B_E(\tilde\tau,\tilde\xi) \lesssim |E| N  \check N^{-2} (L\lor \check L).
\]
Finally, from \eqref{meas_A} and \eqref{Est_1} we conclude.

\medskip

\noindent
\emph{Estimate for high-high to low $J_{HH\to L}$, and for high-high to high frequencies $J_{HH\to H}$.} This is the difficult part of the proof, because for obtaining \eqref{Bilinear} with $s>\frac12$ we do not have the corresponding Carbery-Kenig-Ziesler \cite{CKZ} result. Instead, we will prove that the corresponding smoothing estimate Lemma \ref{asp_lin_estimate_lemma} suffices for the case of negative energies.

\medskip

We prove the most difficult case, the one for  high-high to high frequencies (see below for a comment on the case high-high to low). Here we have $N\sim \tilde N \sim \check N$, where $N\gg 1$.  From \eqref{JNNN} we have to estimate the quantity
\be
\sum_{N\sim \tilde N \sim \check N \gg 1} J_{N,\tilde N,\check N}.
\ee
Note that we also have $|\xi|\sim |\tilde \xi| \sim |\check \xi| \sim |E|^{1/2}N$. Therefore, using \eqref{KernelK} we have
\bee
|K|  & \lesssim &  |E|^{-s/2}  |\xi| N^{-s}  \langle \tilde\sigma\rangle^{-1/2+2\ve}\langle\sigma\rangle^{-1/2-\ve}\langle \check\sigma\rangle^{-1/2-\ve}\\
& \lesssim &  |E|^{-s/2 +1/4^+}  N^{1/2^+-s}  |\xi|^{1/4^-}\langle\sigma\rangle^{-1/2-\ve}|\check\xi|^{1/4^-}\langle \check\sigma\rangle^{-1/2-\ve}.
\eee
Therefore, from \eqref{JNNN}
\bee
|J_{N,\tilde N,\check N}| & \lesssim & |E|^{-s/2 +1/4^+} N^{1/2^+-s} \norm{|\partial_z|^{1/4^-}\mathcal F^{-1}[\langle\sigma\rangle^{-1/2-\ve}  \widehat{P_Nu}]   }_{L^4} \times\\
& & \qquad \times \norm{|\partial_z|^{1/4^-}\mathcal F^{-1}[\langle\check \sigma\rangle^{-1/2-\ve}   \widehat{P_{\check N}v} ] }_{L^4} \|w\|_{L^2}
\eee
From \eqref{Smoothing2} we get
\bee
|J_{N,\tilde N,\check N}| & \lesssim &|E|^{-s/2 +1/4^+}   N^{1/2^+-s} \norm{\mathcal F^{-1}[\langle\sigma\rangle^{-1/2-\ve}  \widehat{P_Nu}]   }_{X^{0,\frac12^-}}  \times \\
&& \qquad \times \norm{\mathcal F^{-1}[\langle\check \sigma\rangle^{-1/2-\ve}   \widehat{P_{\check N}v} ] }_{X^{0,\frac12^-}} \|w\|_{L^2} \\
& \lesssim &  |E|^{-s/2 + 1/4^+ - 0^+}  N^{1/2^+-s} \|P_N u\|_{L^2}\|P_{\check N} v\|_{L^2}\|w\|_{L^2}
\eee
Adding on $N\sim \check N$, we conclude.

\medskip

Finally, some words about the case high-high to low frequencies. In this regime one has $N \sim \check N \gg \tilde N$. Note that we also have $|\xi|\sim |\check \xi| \sim |E|^{1/2}N$, and $ |\tilde \xi| \sim |E|^{1/2} \tilde N$.  Now it is enough to consider the following estimate:
\[
\frac{ | \tilde \xi | \tilde N^s }{ N^s \check N^s } \lesssim \frac{ | \xi | }{ N^s }.
\]
Therefore, using \eqref{KernelK} we have
\bee
|K|  & \lesssim &  |E|^{-s/2}  |\xi| N^{-s}  \langle \tilde\sigma\rangle^{-1/2+2\ve}\langle\sigma\rangle^{-1/2-\ve}\langle \check\sigma\rangle^{-1/2-\ve}\\
& \lesssim &  |E|^{-s/2 +1/4^+}  N^{1/2^+-s}  |\xi|^{1/4^-}\langle\sigma\rangle^{-1/2-\ve}|\check\xi|^{1/4^-}\langle \check\sigma\rangle^{-1/2-\ve},
\eee
and the rest of the proof is similar to the previous case.
\end{proof}

\bigskip

\section{Local well-posedness}
\medskip

In this section we prove Theorem \ref{LWP_neg}. Using the standard Bourgain's method, we will show the following

\begin{thm}\label{Cauchy}
Fix $E<0$. The Cauchy problem for \eqref{NV1}-\eqref{NV2} is locally well-posed in $H^s(\R^2)$ for $s>\frac12$. Moreover, the existence lifetime $T_E>0$ of a solution $v(t)$ with initial data $v_0$ satisfies
\be\label{Existence_E}
T_E\| v_0\|_{H^s}^\al  \gtrsim  |E|^{\al(4s-3)/8},
\ee
for some positive exponent $\al.$ Finally, one has for all $\eta \in [0,1]$ standard cut-off supported in the interval $[-2,2]$, $\eta \equiv 1$ on $[-1,1]$, and $s>\frac12$, and $T<T_E$,
\be\label{Global_local}
\| \eta(t/T) v(t) \|_{X_E^{s,\frac12+\ve}} \lesssim \| v_0\|_{H^s}.
\ee
\end{thm}

\begin{rem}
Note that in this result we give explicit dependence on the energy $E$, for further developments.
\end{rem}

\begin{proof}
The proof is standard, we only check the main lines of the proof. Assume that $v_0\in H^s$. From \eqref{NV1}-\eqref{NV2} we have the local Duhamel's formula for $t\in [0,1]$,
\[
\eta_T(t) v (t)=\eta_T(t) \mathcal T[v](t) := \eta_T(t)  U(t) v_0 + 2\eta_T(t)\int_0^t U(t-t') [\partial_z(vw)+\partial_{\bar z}(v\bar w)]dt',
\]
where $\eta =\eta(t) \in [0,1]$ is a smooth bump function with $\eta(t) =1$ for $t\in[-1,1]$, and $\eta(t) =0$ for $|t| \geq 2$, and $\eta_T(t):= \eta(t/T)$. From this identity, the standard linear estimates (see e.g. Lemma 3.1, Lemma 3.2 of \cite{G}) and \eqref{Bilinear} we have, for any $s>\frac12$,
\bee
\| \eta_T v\|_{X^{s,1/2+\ve}_E} & \leq & \|v_0\|_{H^s} + C \|\eta_T\partial_z(vw)\|_{X^{s,-1/2+ \ve}_E} \\
& \leq& \|v_0\|_{H^s} +  CT^\ve \|\eta_T \partial_z(v w)\|_{X^{s,-1/2+ 2\ve}_E} \\
& \leq & \|v_0\|_{H^s} +  C\frac{T^\ve}{ |E|^{(4s-3)/8} } \|\eta_T v\|^2_{X^{s,1/2+ \ve}_E}
\eee
Now we fix any time $T\sim  \Big(\frac{ |E|^{(4s-3)/8}}{8C \|v_0\|_{H^s}} \Big)^{1/\ve}$, and the ball
\[
\mathcal B:= \Big\{ v \in X^{s,1/2+\ve}_E  \ : \   \| \eta_T v\|_{X^{s,1/2+\ve}_E} \leq  2\|v_0\|_{H^s}  \Big\}.
\]
For $v \in \mathcal B$, one has
\[
\| \eta_T \mathcal T [v]\|_{X^{s,1/2+\ve}_E} \leq   2\|v_0\|_{H^s}.
\]
The contraction property is proved in a similar fashion. The proof is complete.
\end{proof}

\appendix

\section{Proof of \eqref{NV_xy}}\label{A}

We have
\bee
(\partial_z^3 + \partial^3_{\bar z}) v & =&  \frac18(\partial_x -i\partial_y)^3 v + \frac18(\partial_x + i\partial_y)^3 v \\
& =& \frac18 (\partial_x -i\partial_y)^2 (\partial_xv -i \partial_y v) + \frac18 (\partial_x + i\partial_y)^2 (\partial_x v + i \partial_y v) \\
& =& \frac18 (\partial_x -i\partial_y) (\partial_x^2 v -2i \partial_{xy}^2 v  - \partial_y^2 v) + \frac18 (\partial_x + i\partial_y) (\partial_x^2 v + 2i \partial_{xy}^2 v  - \partial_y^2 v) \\
& =& \frac18 (\partial_x^3 v -2i \partial_{xxy}^3 v  - \partial_{xyy}^3 v -i ( \partial_{xxy}^3 v -2i \partial_{xyy}^3 v  - \partial_y^3 v)) \\
&  & + \frac18 (\partial_x^3 v + 2i \partial_{xxy}^3 v  - \partial_{xyy}^3 v + i ( \partial_{xxy}^3 v + 2i \partial_{xyy}^3 v  - \partial_y^3 v)) \\
& =& \frac14\partial_x (\partial_x^2 v  -  3 \partial_y^2 v).
\eee
Similarly, if $w=w_1+iw_2 =(w_1,w_2)$,
\bee
\partial_z w + \partial_{\bar z}\bar w &=& \frac12 (\partial_x w-i\partial_yw) +\frac12(\partial_x \bar w +i \partial_y\bar w) \\
& =&  \partial_x w_1+  \partial_y w_2 \\
& =&  \nabla \cdot w.
\eee
Similarly,
\bee
2\partial_z (vw) +2 \partial_{\bar z}(v \bar w) & =& 4\re  \partial_z (vw)\\
& =& 2\re (\partial_x -i\partial_y) (v(w_1+iw_2)) \\
& =&  2\partial_x  (v w_1) +2 \partial_y (v w_2) \\
& =&  2\nabla .(vw).
\eee
On the other hand, \eqref{NV2} can be written as
\[
\partial_y w_{1}+\partial_x w_{2} =3\partial_y v, \quad \partial_y w_{2}- \partial_x w_{1} =3 \partial_x v,
\]
Note that one can recover $w_1$ and $w_2$ from $v$ via the identities
\[
\Delta w_1 = 3( \partial_y^2 v -\partial_x^2 v), \quad \Delta w_2 = 6\partial_{xy} v.
\]
Therefore equation \eqref{NV1} reads now
\[
\partial_t v = 2[ \partial_x (\partial_x^2 v-3\partial_y^2 v) +\nabla .(vw) -E\nabla .w],
\]
as desired.

\section{High-energy limit of NV equations}\label{B}

Consider the Novikov-Veselov system of equations written in $ ( x, y ) $ variables:
\begin{equation}
\label{NV}
 \begin{cases}
  & \partial_{t} v = 2 \partial_{xxx} v - 6 \partial_{xyy} v + 2[ \partial_{x}(v w_1) +\partial_{y}(vw_2) ] - 2 E ( \partial_{x} w_1 + \partial_{y} w_2 ), \\
  & \partial_{x} w_1 - \partial_{y} w_2 = - 3 \partial_{x} v, \\
  & \partial_{x} w_2 + \partial_{y} w_1 = 3 \partial_{y} v.
 \end{cases}
\end{equation}
(see also (\ref{NV_xy})).

Put $ E = \pm \kappa^2 $ and dilate the second variable: $ y = \kappa Y $. Then system (\ref{NV}) becomes
  \begin{subequations}
  \label{NV_dilated}
  \begin{align}
  \label{dilated1}
  & \partial_{t} v = 2 \partial_{xxx} v - 6 \frac{ 1 }{ \kappa^2 } \partial_{xYY} v + 2 \partial_{x}(v w_1) + 2 \frac{ 1 }{ \kappa }\partial_{Y}(vw_2) - 2 E \partial_{x} w_1 - 2 E \frac{ 1 }{ \kappa } \partial_{Y} w_2, \\
  \label{dilated2}
  & \partial_{x} w_1 - \frac{ 1 }{ \kappa } \partial_{Y} w_2 = - 3 \partial_{x} v, \\
  \label{dilated3}
  & \partial_{x} w_2 + \frac{ 1 }{ \kappa } \partial_{Y} w_1 = 3 \frac{ 1 }{ \kappa } \partial_{Y} v.
  \end{align}
  \end{subequations}
Formally, let us search the solution of (\ref{NV_dilated}) in the following form:
\begin{align}
 & v = v_0( x, Y ), \\
 \label{w_one}
 & w_1 = \mp 3 \kappa^2 + w_1^0( x, Y ) + w_1^1( x, Y ) \frac{ 1 }{ \kappa } + w_1^2( x, Y ) \frac{ 1 }{ \kappa^2 }, \\
 \label{w_two}
 & w_2 = w_2^0( x, Y ) + w_2^1( x, Y ) \frac{ 1 }{ \kappa } + w_2^2( x, Y ) \frac{ 1 }{ \kappa^2 }.
\end{align}
We will also suppose that
\begin{equation}
\label{assumption}
 w_j^k( x, Y ) \to 0 \quad \text{ as } \quad x \to \infty, \quad j = 1,2, \quad k = 0, 1, 2.
\end{equation}
Formulas (\ref{w_one}), (\ref{w_two}) are approximations, up to order $ -2 $, of asymptotic expansions with respect to large $ \kappa $ of solutions to equations (\ref{dilated2})-(\ref{dilated3}) such that $ w_1 \pm 3 \kappa^2 $, $ w_2 $  are localized functions.

Equating terms of order $ 0 $ in $ \kappa $ in equations (\ref{dilated2}), (\ref{dilated3}) and taking into account condition (\ref{assumption}), we obtain
\begin{equation}
 w_1^0 = - 3 v_0, \quad w_2^0 = 0.
\end{equation}

Equating terms of order $ -1 $ in $ \kappa $ in equations (\ref{dilated2}), (\ref{dilated3}) and taking into account condition (\ref{assumption}), we obtain
\begin{equation}
 w_1^1 = 0, \quad w_2^1 = 6 \partial_{x}^{-1} \partial_{Y} v_0.
\end{equation}

Finally, equating terms of order $ -2 $ in $ \kappa $ in equations (\ref{dilated2}), (\ref{dilated3}) and taking into account condition (\ref{assumption}), we get
\begin{equation}
 w_1^2 = 6 ( \partial_{x}^{-1} )^2 \partial_{Y}^2 v_0, \quad w_2^2 = 0.
\end{equation}

Thus we have
\begin{subequations}
\label{ansatz}
\begin{align}
 & v = v_0( x, Y ), \\
 & w_1 = \mp 3 \kappa^2 - 3 v  + 6 ( \partial_{x}^{-1} )^2 \partial_{Y}^2 v_0 \frac{ 1 }{ \kappa^2 }, \\
 & w_2 = 6 \partial_{x}^{-1} \partial_{Y} v_0 \frac{ 1 }{ \kappa }.
\end{align}
\end{subequations}
Now let us insert this ansatz into equation (\ref{dilated1}). It is easy to verify that the coefficient of the term of order $ 2 $ in $ \kappa $ is equal to zero, and there are no terms of order $ 1 $. Finally, equating terms of order $ 0 $ in $ \kappa $, we obtain the following equation for $ v_0 $:
\begin{equation}
\label{KP}
 \partial_{t} v_0 = 2 \partial_{xxx} v_0 - 12 v \partial_{x} v_0 \mp 24 \partial_{x}^{-1} \partial_{Y}^2 v_0.
\end{equation}

Note that if $ v_0( x, Y, t ) $ is a solution of equation (\ref{KP}), then
\begin{equation}
\label{uv_relation}
u( x, y, t ) = - v_0( -x, 2 y, \frac{ 1 }{ 2 } t )
\end{equation} is a solution of KP in its classical form
\begin{equation}
\label{KP_classic}
\partial_t u + 6 u \partial_x u + \partial_{ xxx } u = \pm 3 \partial_{ x }^{ -1 } \partial_{ y }^2 u.
\end{equation}

Thus, at high energy $ E $ limit, after an appropriate dilation of $ y $ variable, equation $ NV_+ $ becomes KPI and equation $ NV_- $ becomes KPII.

\end{document}